\documentclass[dvips,aap,preprint]{imsart}

\RequirePackage[OT1]{fontenc} \RequirePackage{amsthm,amsmath}
\RequirePackage[numbers]{natbib}
\RequirePackage[colorlinks,citecolor=blue,urlcolor=blue]{hyperref}
\RequirePackage{hypernat}

\numberwithin{equation}{section} \theoremstyle{plain}

\endlocaldefs

\usepackage{enumerate}

\usepackage{amsmath}
\usepackage{amssymb}
\usepackage{color}
\usepackage{url}
\usepackage{graphicx}
\usepackage{hyperref}
\usepackage{pstricks}
\usepackage{epsfig}
\usepackage{pst-grad} 
\usepackage{pst-plot} 

\newcommand{\tr}{\mbox{\textnormal{tr\,}}}
\newcommand{\vech}{\mbox{\textnormal{vech}}}

\newcommand{\E}{\mbox{\textnormal{E}}}
\newcommand{\PP}{\mathbb{P}}
\newcommand{\R}{\mathbb{R}}
\newcommand{\Q}{\mathbb{Q}}
\newcommand{\Z}{\mathbb{Z}}
\newcommand{\N}{\mathbb{N}}
\newcommand{\K}{\mathbb{K}}

\newcommand{\C}{\mathbb{C}}

\newcommand{\half}{\mbox{\normalsize${\frac{1}{2}}$}}
\newcommand{\quart}{\mbox{\normalsize${\frac{1}{4}}$}}
\newcommand{\dd}{\textnormal{d}}
\newcommand{\ii}{\textnormal{i}}

\newcommand{\wX}{\widetilde{X}}

\newcommand{\hK}{\widehat{K}}

\newcommand{\hb}{\widehat{b}}

\newcommand{\wb}{\widetilde{b}}
\newcommand{\wB}{\widetilde{B}}

\newcommand{\wc}{\widetilde{c}}
\newcommand{\wnu}{\widetilde{\nu}}

\newcommand{\wK}{\widetilde{K}}

\newtheorem{theorem}{Theorem}[section]
\newtheorem{lemma}[theorem]{Lemma}

\newtheorem{prop}[theorem]{Proposition}
\newtheorem{cor}[theorem]{Corollary}

\theoremstyle{definition}
\newtheorem{definition}[theorem]{Definition}
\newtheorem{example}[theorem]{Example}

\theoremstyle{remark}
\newtheorem{remark}[theorem]{Remark}

\numberwithin{equation}{section}

\begin{document}

\begin{frontmatter}
\title{The affine transform formula for affine jump-diffusions with a general closed convex state space}
\runtitle{The affine transform formula}

\begin{aug}
\author{\fnms{Peter} \snm{Spreij}\ead[label=e1]{spreij@uva.nl}}
and
\author{\fnms{Enno} \snm{Veerman}\ead[label=e2]{e.veerman@uva.nl}}

\runauthor{P. Spreij and E. Veerman}

\affiliation{University of Amsterdam}

\address{Korteweg-de Vries Institute for
Mathematics\\
Universiteit van Amsterdam \\
Science park 904\\
1098XH Amsterdam
\\The Netherlands\\
\printead{e1}  \\ \phantom{E-mail:\ }\printead*{e2}}

\end{aug}

\begin{abstract}
We establish existence of exponential moments and the validity of
the affine transform formula for affine jump-diffusions with a
general closed convex state space.  This extends known results for
affine jump-diffusions with a canonical state space. The key step
is to prove the martingale property of an exponential local
martingale, using the well-posedness of the associated martingale
problem. By analytic extension we obtain the affine transform
formula for complex exponentials, in particular for the
characteristic function. Our results apply to a wide class of
affine processes, including those with a matrix-valued state
space, which have recently gained interest in the literature.
\end{abstract}

\begin{keyword}[class=MSC]
\kwd[Primary ]{60J60} \kwd{91G30}
\end{keyword}

\begin{keyword}
\kwd{affine jump-diffusions} \kwd{generalized Riccati equations}
\kwd{exponential moments} \kwd{exponential martingales}
\kwd{affine transform formula}
\end{keyword}

\end{frontmatter}

\section{Introduction}

Affine jump-diffusions, as introduced in \cite{dk96,dps00}, are
widely used in finance, due to their flexibility and mathematical
tractability.
Their main attraction  lies in the so-called \emph{affine
transform formula}
\begin{align}\label{eq:affinetransform}
 \E_x \exp(u^\top
X_t)=\exp(\psi_0(t,u)+\psi(t,u)^\top x),\quad u\in\C^p,X_0=x,
\end{align}
which relates exponential moments of the affine jump-diffusion $X$
to solutions $(\psi_0,\psi)$ to certain ordinary differential
equations, called generalized Riccati equations. The importance of
this formula is particularly elucidated in option and bond
pricing. For example, the affine transform formula yields a closed
form expression for the zero-coupon bond price in an affine term
structure model, see \cite{dk96,dps00}. Moreover, taking $u$
purely imaginary in (\ref{eq:affinetransform}) gives the
characteristic function of $X_t$, which is of vital importance for
calculating more general prices by using Fourier methods, e.g.\
those of \cite{carr99}.

The validity of the affine transform formula is not
straightforward in general. In the literature most results in this
respect
are proved for affine jump-diffusions living on the state space
$\R^m_+\times\R^{p-m}$, see
\cite{dfs03,fm09,glassermankim,kallkarb08,lev04} amongst others.
This state space, often called the \emph{canonical state space},
was introduced in \cite{ds00} and has traditionally been the
standard choice in financial applications. Currently though, there
is a growing number of papers devoted to matrix-valued affine
processes living on $S^p_+$, the cone of positive semi-definite
matrices, or on variations of it, like $S^p_+\times\R$, see for
instance \cite{cfmt09,fons08,gjs09,grass08,karbpfafstelz10}.
Moreover, in an accompanying paper \cite{part2} we provide further
examples of affine diffusions with a ``non-canonical'' state
space, e.g.\ those with a quadratic state space, indicating that
this class is rather rich. This feeds the demand to obtain results
for the validity of (\ref{eq:affinetransform}) for more general
state spaces than $\R^m_+\times\R^{p-m}$, which is the scope of
the present paper.

We highlight that one of our aims is to establish for
\emph{arbitrary} state spaces the affine transform formula for the
characteristic function, a crucial feature for the application of
affine processes in mathematical finance as pointed out in the
first paragraph. To our knowledge, this important property has
only been derived for affine processes living on a canonical state
space, see \cite{dfs03,fm09}. The complicated factor is that the
so-called \emph{admissibility conditions} that are required for
stochastic invariance and for existence and uniqueness of the
affine process, are much more involved for a non-canonical than
for a canonical state space, due to the curvedness of the
boundary. As a consequence, it is much harder for general state
spaces to control the solutions of the Riccati equations by means
of these admissibility conditions. We circumvent this difficulty
by relying on probabilistic methods instead.


%

The contents and set-up of the paper are as follows. First we
derive a general result in Section~\ref{sec:prelim} on the
martingale property of a stochastic exponential, building on
results in \cite{cherfilyor}. Next we apply this in
Section~\ref{sec:affineSDEs} to the stochastic exponential of
affine jump-diffusions in order to obtain sufficient conditions on
$\psi$ such that (\ref{eq:affinetransform}) holds, irrespective of
the underlying state space. This is our first main result and
extends the result in \cite{kallkarb08}, which is limited to the
canonical state space.

Our second main result concerns the full range of validity of
(\ref{eq:affinetransform}) for affine jump-diffusions with an
arbitrary closed convex state space, under some moment conditions
on the jump-measure. We show existence of solutions to the Riccati
equations under finiteness of exponential moments and establish
the affine transform formula (\ref{eq:affinetransform}) whenever
either side of (\ref{eq:affinetransform}) is well-defined, both
for real and complex $u$. This generalizes a recent result by
\cite{fm09}, which concerns affine diffusions on the canonical
state space $\R^m_+\times\R^{p-m}$ under absence of jumps.

The proof of the second main result is distributed over two
sections. In Section~\ref{sec:realexp} we establish the full range
of validity for real-valued exponentials, while in
Section~\ref{sec:complexexp} we extend this to complex ones. For
the latter we use the analyticity of both the characteristic
function and the solutions to the Riccati equations. A
complicating matter is that an affine jump-diffusion with a
general state space is in general not \emph{infinite divisible},
as opposed to those with a canonical state space. Hence, a priori
it is not excluded that the left-hand side of
(\ref{eq:affinetransform}) vanishes for certain complex $u$, which
would yield an explosion of $\psi$. We tackle this problem by
using properties of analytic functions.

In Section~\ref{sec:boundedexp} we relax the moment conditions on
the jump-measure and establish the validity of (a slight variation
of) (\ref{eq:affinetransform}) in the case the left-hand side is
uniformly bounded in $x$ and $t$, which includes the
characteristic function. This yields our third main result and it
enables us to obtain sufficient conditions for infinite
divisibility in Subsection~\ref{subsec:infdiv} as well as proving
additional results for the case that the state space is a
self-dual cone in Subsection~\ref{subsec:selfdual}.

Finally, some technical results used throughout the text are put
in the appendix, in order to keep a fluid presentation.

\section{Preliminary result on exponential martingales}\label{sec:prelim}
 In this section we obtain sufficient conditions for the
 martingale property of a
stochastic exponential. This is the key-ingredient in obtaining
our results concerning the affine transform formula for affine
jump-diffusions in the next sections. We use the framework of
\cite{cherfilyor} with some slight modifications and derive a
corollary of its main result, \cite[Theorem~2.4]{cherfilyor}, in
Theorem~\ref{th:expmart}.

Let $E\subset\R^p$ be a closed set and $E_\Delta=E\cup\{\Delta\}$
the one-point compactification of $E$. Every measurable function
$f$ on $E$ is extended to $E_\Delta$ by setting $f(\Delta)=0$.
Throughout this section, $\Omega$ denotes a subset of
$D_{E_\Delta}[0,\infty)$, the space of c\`{a}dl\`{a}g functions
$\omega:[0,\infty)\rightarrow E_\Delta$. Unless mentioned
otherwise, $\Omega$ is equipped with the $\sigma$-algebra
$\mathcal{F}^X=\sigma(X_s:s\geq0)$ and filtration
$\mathcal{F}^X_t:=\sigma(X_s:0\leq s\leq t)$, generated by the
coordinate process $X$ given by $X_t(\omega)=\omega(t)$.

Let us be given measurable functions $b:E\rightarrow\R^p$,
$c:E\rightarrow S_+^p$ (space of positive semi-definite $(p\times
p)$-matrices) and a transition kernel $K$ from $E$ to
$F\subset\R^p\backslash\{0\}$ such that $E+F\subset E$. Assume
that
\begin{align}\label{eq:condbcK}
b(\cdot),c(\cdot) \mbox{ and }\int (|z|^2\wedge |z|) \dd
K(\cdot,\dd z)\mbox{ are bounded on compacta of $E$},
\end{align}
and
\begin{align}\label{eq:condgrowK}
\int_{\{|z|>1\}}|z|^q K(x,\dd z)\leq C(1+|x|^q),\mbox{ for some
$C,q>0$, all $x\in E$}.
\end{align}
Write $\nabla f$ for the gradient of $f$ (as a row vector) and
$\nabla^2 f$ for the Hessian. Then
\begin{equation}\label{eq:opA}
\begin{split}
\mathcal{A}f(x)&=\nabla f(x) b(x)+\half\tr(\nabla^2
f(x)c(x))\\&+\int (f(x+z)-f(x)-\nabla f(x) z)K(x,\dd z)
\end{split}
\end{equation}
defines a linear operator $\mathcal{A}:C^\infty_c(E)\rightarrow
B(E)$, see Lemma~\ref{lem:appendix} in the appendix. Here,
$C^\infty_c(E)$ denotes the space of $C^\infty$-functions on $E$
with compact support and $B(E)$ the space of bounded measurable
functions on $E$.
\begin{definition}
A probability measure $\PP$ on $(\Omega,\mathcal{F}^X)$ is called
a solution of the martingale problem for $\mathcal{A}$ if
\begin{align}\label{eq:Mft}
M^f_t=f(X_t)-f(X_0)-\int_0^t \mathcal{A}f(X_s)\dd s
\end{align}
is a $\PP$-martingale with respect to $(\mathcal{F}^X_t)$ for all
$f\in C^\infty_c(E)$. If in addition $\lambda$ is a probability
measure on $E$ such that $\PP\circ X_0^{-1}=\lambda$, then we say
$\PP$ is a solution of the martingale problem for
$(\mathcal{A},\lambda)$ and we often write $\PP=\PP_\lambda$. If
$\lambda=\delta_x$, the Dirac-measure at $x$ for some $x\in E$,
then we write $\PP_x$ instead. Likewise, $\E_\lambda$ denotes the
expectation with respect to $\PP_\lambda$ and $\E_x$ the
expectation with respect to $\PP_x$. We call the martingale
problem for $\mathcal{A}$ well-posed if for all $x\in E$ there
exists a unique solution $\PP_x$ on
$(D_E[0,\infty),\mathcal{F}^X)$ of the martingale problem for
$(\mathcal{A},\delta_x)$.
\end{definition}

\begin{remark}\label{remark} 1. In case $\Omega=D_E[0,\infty)$, then it holds that $\PP$ is a
solution of the martingale problem for $\mathcal{A}$ on
$(\Omega,\mathcal{F}^X)$ if and only if $X$ is a special
jump-diffusion on
$(\Omega,\mathcal{F}^X,(\mathcal{F}^X_{t+}),\PP)$ with
differential characteristics $(b(X),c(X),K(X,\dd z))$, by
\cite[Theorem~II.2.42]{JacShir} and a modification of
\cite[Proposition~3.2]{cherfilyor}. In that case, $X$ can be
decomposed according to its characteristics by
\begin{align}\label{eq:decompx}
X=X_0+B+X^c+z\ast(\mu^X-\nu^X),
\end{align}
where $B_t=\int_0^t b(X_s)\dd s$, $\mu^X$ is the random measure
associated to the jumps of $X$, $\nu^X(\dd t,\dd z)=K(X_t,\dd
z)\dd t$ its compensator and $X^c$ is the continuous local
martingale part of $X$ with quadratic variation $\langle
X^c\rangle_t=\int_0^t c(X_s)\dd s$.

2. If the martingale problem for $\mathcal{A}$ is well-posed, then
$(\PP_x)_{x\in E}$ is a transition kernel and for all probability
measures $\lambda$ on $E$ it holds that $\PP_\lambda=\int
\PP_x\lambda(\dd x)$ is the unique solution of the martingale
problem for $(\mathcal{A},\lambda)$. In addition, the strong
Markov property holds, i.e.\
\[
\E_\lambda
(f(X_{t+\tau})|\mathcal{F}_t^X)=\E_{X_{\tau}}f(X_{t}),\quad \mbox{
$\PP_\lambda$-a.s.}
\]
for all integrable $f$, $t\geq0$ and a.s.\ finite
$(\mathcal{F}^X_t)$-stopping times $\tau$. See the appendix for
the proof of this assertion.

3. If for some $x_0\in E$, $\PP$ is a solution of the martingale
problem for $(\mathcal{A},\delta_{x_0})$, then $
\mathcal{A}f(x_0)=\lim_{t\downarrow 0}(\E f(X_t)-f(x_0))/t,\mbox{
for $f\in C^\infty_c(E)$}$. This follows by taking expectations in
(\ref{eq:Mft}) and applying Fubini, which is justified since
$\mathcal{A}f$ is bounded.
\end{remark}

In addition to $b$, $c$ and $K$, let us be given a measurable
function $\wb:E\rightarrow\R^p$ and a transition kernel $\wK$ from
$E$ to $F$. Assume that
\begin{align}\label{eq:condcontinu}
\wb(\cdot),c(\cdot)\mbox{ are continuous, }(|z|^2\wedge
|z|)\wK(\cdot,\dd z)\mbox{ is weakly continuous},
\end{align}
and
\begin{align}\label{eq:condmoregrowK}
\int_{\{|z|>1\}} |z|^q\log|z| \wK(x,\dd z)\leq C(1+|x|^q),\mbox{
some $C,q>0$, all $x\in E$}.
\end{align}
Then
\begin{equation}\label{eq:opwA}
\begin{split}
\widetilde{\mathcal{A}}f(x)&=\nabla f(x)\wb(x)+\half\tr(\nabla^2
f(x)c(x))\\&+\int (f(x+z)-f(x)-\nabla f(x) z)\wK(x,\dd z)
\end{split}
\end{equation}
defines a linear operator
$\widetilde{\mathcal{A}}:C^\infty_c(E)\rightarrow C_0(E)$, where
$C_0(E)$ denotes the space of continuous functions on $E$
vanishing at infinity, see Lemma~\ref{lem:appendix}. Here, weak
continuity means that $x\mapsto \int f(z) (|z|^2\wedge
|z|)\wK(x,\dd z)$ is continuous for all $f\in C_b(F)$, the space
of bounded continuous functions on $F$. As in \cite{cherfilyor},
we assume there exist measurable mappings $h:E\rightarrow\R^p$,
$w:E\times F\rightarrow (-1,\infty)$ such that $\wb$ and $\wK$ are
related to $b$ and $K$ by
\begin{equation}\label{eq:wbwK}
\begin{split}
\wb(x)&=b(x)+c(x)h(x)+\int z w(x,z) K(x,\dd z)\\
\wK(x,\dd z)&= (w(x,z)+1) K(x,\dd z).
\end{split}
\end{equation}

Our aim is to show the martingale property of a stochastic
exponential with the aid of \cite[Theorem~2.4]{cherfilyor}, under
the assumption that the martingale problem for $\mathcal{A}$ is
well-posed. This requires the existence of a solution of the
martingale problem for $\widetilde{\mathcal{A}}$ on
$(D_E[0,\infty),\mathcal{F}^X)$, which is part of the assumptions
in \cite[Theorem~2.4]{cherfilyor}. In our case though, we are able
to \emph{derive} the existence by invoking
\cite[Theorem~4.5.4]{ethier86}, as the range of
$\widetilde{\mathcal{A}}$ is contained in $C_0(E)$, due to the
additional continuity conditions (\ref{eq:condcontinu}). Note that
these conditions are similar as those in
\cite[Theorem~2.2]{str75}, where existence is derived for the case
$E=\R^p$.

The next lemma will be used to obtain the \emph{maximum principle}
for $\widetilde{\mathcal{A}}$ in the ensuing proposition, where we
establish the existence of a solution of the martingale problem
for $\widetilde{\mathcal{A}}$.
\begin{lemma}\label{lem:posmaxprinc}
Let $x_0\in E$ and suppose the martingale problem for
$(\mathcal{A},\delta_{x_0})$ has a solution $\PP$ on
$(\Omega,\mathcal{F}^X)$ with $\Omega=D_E[0,\infty)$. Suppose
$f\in C^\infty_c(E)$ attains its maximum at $x_0$. Then it holds
that
\begin{enumerate}
\item $\nabla f(x_0) c(x_0)=0$,
\item $\int\nabla f(x_0) z K(x_0,\dd z)$ is
well-defined and finite,
\item $\nabla f(x_0) b(x_0) -\int \nabla f(x_0) z K(x_0,\dd
z)+\half\tr(\nabla^2 f(x_0)c(x_0))\leq 0$. \end{enumerate}
\end{lemma}
\begin{proof}
By Remark~\ref{remark} part 1, $X$ is a jump-diffusion on
$(\Omega,\mathcal{F}^X,\mathcal{F}^X_{t+},\PP)$ with differential
characteristics $(b(X),c(X),K(X,\dd z))$. Let $\lambda\in\R^p$ and
$\varepsilon>0$ be arbitrary, define $h(x)=\lambda 1_{\{x=x_0\}}$
and $w(x,z)=(\varepsilon-1) 1_{\{x=x_0\}\cap\{|z|>\varepsilon\}}$
and write $H_t=h(X_t)$, $W(t,z)=w(X_t,z)$ and
\[
Z=H\cdot X + W\ast (\mu^X-\nu^X).
\]
For $T>0$ it holds that $\mathcal{E}(Z)^T=\mathcal{E}(Z^T)$ is a
uniformly integrable martingale by \cite[Theorem~IV.3]{lepmem78},
since
\begin{align*}
&\half \langle Z^c\rangle_T + ((W+1)\log(W+1)-W)\ast
\nu^X_T\\&=\int_0^T (\half \lambda^\top c(x_0)
\lambda+\int_{\{|z|>\varepsilon\}}(\varepsilon\log\varepsilon-\varepsilon+1)K(x_0,\dd
z)))1_{\{X_s=x_0\}}\dd s
\end{align*}
has finite expectation as it is bounded. By Girsanov's Theorem
\cite[Proposition~4]{kall06}, $\Q=\mathcal{E}(Z)_T\cdot \PP$ is a
probability measure on $\mathcal{F}^X$ equivalent to $\PP$ and $X$
is a special jump-diffusion on $[0,T]$ with differential
characteristics $(\hb(X_t),c(X_t),\hK(X_t,\dd z))$ under $\Q$
given by
\begin{align*}
\hb(x)&=b(x)+c(x)h(x)+\int z w(x,z) K(x,\dd z)\\
\hK(x,\dd z)&= (w(x,z)+1) K(x,\dd z).
\end{align*}
Therefore, \cite[Theorem~II.2.42]{JacShir} yields that $\Q$ is a
solution of the martingale problem for
$(\widehat{\mathcal{A}},\delta_{x_0})$ on $(\Omega,\mathcal{F}^X)$
with time restricted to $[0,T]$, with the linear operator
$\widehat{\mathcal{A}}:C^\infty_c(E)\rightarrow B(E)$ defined by
\begin{align*}
\widehat{\mathcal{A}}f(x)&=\nabla f(x)\hb(x)+\half\tr(\nabla^2
f(x)c(x))\\&+\int(f(x+z)-f(x)-\nabla f(x) z)\hK(x,\dd z)\\
&=\mathcal{A}f(x)+\nabla f(x)
c(x)h(x)+\int(f(x+z)-f(x))w(x,z)K(x,\dd z).
\end{align*}
Hence $\widehat{\mathcal{A}}f(x_0)$ equals
\begin{align}\label{eq:afx0}
\mathcal{A}f(x_0)+\nabla f(x_0) c(x_0)
\lambda+\int_{\{|z|>\varepsilon\}}(f(x_0+z)-f(x_0))(\varepsilon-1)K(x_0,\dd
z).
\end{align}
Since $f$ attains its maximum at $x_0$, Remark~\ref{remark} part 3
yields that $\widehat{\mathcal{A}}f(x_0)\leq 0$. Therefore,
(\ref{eq:afx0}) is non-positive for all $\lambda\in\R^p$ and
$\varepsilon>0$. This yields that $\nabla f(x_0) c(x_0)=0$, which
is the first assertion. It follows that
\begin{align}\label{eq:afx1}
\mathcal{A}f(x_0)+\int_{\{|z|>\varepsilon\}}(f(x_0+z)-f(x_0))(\varepsilon-1)K(x_0,\dd
z)\leq 0,
\end{align}
for all $\varepsilon>0$. Letting $\varepsilon\downarrow0$ in
(\ref{eq:afx1}) and applying the Monotone Convergence Theorem
gives
\[
\mathcal{A}f(x_0) -\int(f(x_0+z)-f(x_0))K(x_0,\dd z)\leq 0.
\]
The left-hand side equals
\[
\nabla f(x_0) b(x_0) -\int \nabla f(x_0) z K(x_0,\dd
z)+\half\tr(\nabla^2 f(x_0)c(x_0)),
\]
which yields the second and third assertion.
\end{proof}

\begin{prop}\label{prop:1}
Suppose for all $x\in E$ there exists a solution of the martingale
problem for $(\mathcal{A},\delta_x)$ on
$(D_E[0,\infty),\mathcal{F}^X)$. Then for all $x\in E$ there
exists a solution of the martingale problem for
$(\widetilde{\mathcal{A}},\delta_x)$ on $\Omega$ given by
\begin{align}\label{eq:omega}
\Omega=\{\omega\in D_{E_\Delta}[0,\infty):\mbox{if }
\omega(t-)=\Delta\mbox{ or }\omega(t)=\Delta\mbox{ then
}\omega(s)=\Delta\mbox{ for }s\geq t\}.
\end{align}
\end{prop}
\begin{proof}
We check the conditions of \cite[Theorem~4.5.4]{ethier86}. Let
$f\in C^\infty_c(E)$ attain its maximum at some point $x_0\in E$.
By Lemma~\ref{lem:posmaxprinc}, we can write
$\widetilde{\mathcal{A}}f(x_0)$ as the sum of two non-positive
terms, namely
\[
\nabla f(x_0) b(x_0) -\int \nabla f(x_0) z K(x_0,\dd
z)+\half\tr(\nabla^2 f(x_0)c(x_0))
\]
and
\[\int (f(x_0+z)-f(x_0))(w(x_0,z)+1)K(x_0,\dd z).
\]
Hence $\widetilde{\mathcal{A}}f(x_0)\leq 0$. This yields that
$\widetilde{\mathcal{A}}$ satisfies the (positive) maximum
principle. Since $\widetilde{\mathcal{A}}:C_c^\infty(E)\rightarrow
C_0(E)$ and $C_c^\infty(E)$ is dense in $C_0(E)$,
\cite[Theorem~4.5.4]{ethier86} yields for all $x\in E$ the
existence of a solution $\PP_x$ of the martingale problem for
$(\widetilde{\mathcal{A}},\delta_x)$ on
$(D_{E_\Delta}[0,\infty),\mathcal{F}^X)$. In order to obtain a
solution on $\Omega$, we define the stopping time
\begin{align}\label{eq:TDelta}
T_\Delta=\inf\{t\geq 0:X_{t-}=\Delta\mbox{ or }X_t=\Delta\},
\end{align}
and write $X'=X^{T_\Delta}$. Then $X'(\omega)\in\Omega$ for all
$\omega\in D_{E_\Delta}[0,\infty)$ and for all $f\in
C_c^\infty(E)$ it holds that (recall
$\widetilde{\mathcal{A}}f(\Delta)=0$)
\begin{align*}
f(X'_t)-f(X'_0)-\int_0^t \widetilde{\mathcal{A}}f(X'_s)\dd
s&=f(X^{T_\Delta}_t)-f(X^{T_\Delta}_0)-\int_0^{t\wedge T_\Delta}
\widetilde{\mathcal{A}}f(X_s)\dd s\\&=(M^f)^{T_\Delta}_t,
\end{align*}
where $M^f$ is given by (\ref{eq:Mft}) with $\mathcal{A}$ replaced
by $\widetilde{\mathcal{A}}$. Since $M^f$ is a right-continuous
$\PP_x$-martingale on $(\mathcal{F}^X_t)$ for $f\in
C^\infty_c(E)$, $(M^f)^{T_\Delta}$ is a martingale on
$(\mathcal{F}^{X'}_t)$. Hence $\PP_x\circ (X')^{-1}$ is a solution
of the martingale problem for $(\widetilde{\mathcal{A}},\delta_x)$
on $(\Omega,\mathcal{F}^{X'})$ for all $x\in E$, as we needed to
show.
\end{proof}

\begin{prop}\label{prop:2}
Let $x_0\in E$ and suppose there exists a solution $\PP$ of the
martingale problem for $(\widetilde{\mathcal{A}},\delta_{x_0})$ on
$(\Omega,\mathcal{F}^X)$ with $\Omega$ given by (\ref{eq:omega}).
Assume the growth condition
\begin{align}\label{eq:growth}
|\wb(x)|^2+|c(x)|+\int  |z|^2 \wK(x,\dd z)\leq C(1+|x|^2),\mbox{
some $C>0$, all $x\in\R^p$}.
\end{align}
Then it holds that $\PP(X\in D_{E}[0,\infty))=1$.
\end{prop}
\begin{proof}
By the remark preceding \cite[Proposition~3.2]{cherfilyor}, a
transition to $\Delta$ can only occur by explosion. Define
stopping times
\[
T_n=\inf\{t\geq 0:|X_{t-}|\geq n\mbox{ or }|X_t|\geq n\}\wedge n.
\]
By \cite[Proposition~3.2]{cherfilyor}, $X^{T_n}$ is a special
semimartingale with differential characteristics
$(\wb(X^{T_n})1_{[0,T_n]},\wc(X^{T_n})1_{[0,T_n]},\wK(X^{T_n},\dd
z)1_{[0,T_n]})$. Lemma~\ref{lem:sup} yields
\[
\E \sup_{t\leq T\wedge T_n}|X_t|\leq C(T)<\infty,
\]
for all $T>0$, with $C(T)$ a positive constant that does not
depend on $n$. Letting $n\rightarrow\infty$ we get
\[
\E \sup_{t\leq T\wedge T_\Delta}|X_t|<\infty,
\]
for all $T>0$, where $T_\Delta$ is given by (\ref{eq:TDelta}).
Hence $T_\Delta>T$ almost surely for all $T$. This proves the
assertion.
\end{proof}
Having derived the existence of a solution of the martingale
problem for $\widetilde{\mathcal{A}}$ from the existence of a
solution for ${\mathcal{A}}$, we are now ready to prove the
martingale property of a stochastic exponential by the use of
\cite[Theorem~2.4]{cherfilyor}.
\begin{theorem}\label{th:expmart}
Suppose (\ref{eq:growth}) holds and
\begin{equation}\label{eq:condbcK2}
\begin{split}
&x\mapsto h(x)^\top c(x) h(x)\mbox{ and }x\mapsto\int
(w(x,z)-\log(w(x,z)+1))K(x,\dd z)\\&\mbox{ are bounded on
compacta}.
\end{split}
\end{equation}
Let $\Omega=D_E[0,\infty)$, write $H_t=h(X_t)$, $W(t,z)=w(X_t,z)$
and suppose $\PP$ is a solution of the martingale problem for
$\mathcal{A}$ on $(\Omega,\mathcal{F}^X)$, which yields the
decomposition (\ref{eq:decompx}) for $X$. If the martingale
problem for $\mathcal{A}$ is well-posed, then
\[
L=\mathcal{E}(H\cdot X^c+W\ast(\mu^X-\nu^X))
\]
is an $((\mathcal{F}^X_{t+}),\PP)$-martingale and the martingale
problem for $\widetilde{\mathcal{A}}$ is well-posed.
\end{theorem}
\begin{proof}
First assume $\PP=\PP_x$ is a solution of the martingale problem
for $(\mathcal{A},\delta_x)$ for some $x\in E$. By
Proposition~\ref{prop:1} and \ref{prop:2}, there exists a solution
$\Q_x$ of the martingale problem for
$(\widetilde{\mathcal{A}},\delta_x)$ on $(\Omega,\mathcal{F}^X)$.
We can apply \cite[Theorem~2.4]{cherfilyor} with the roles of
$(\mathcal{A}, \PP)$ and $(\widetilde{\mathcal{A}},\Q)$ reversed.
Indeed, in the notation of \cite{cherfilyor} we have $\phi_1=-h$,
$\phi_2=0$, $\phi_3=1/(w+1)$ and these functions satisfy the
criterion mentioned in \cite[Remark~2.5]{cherfilyor} by the
assumptions. This yields $\left.\PP_x\right|_{\mathcal{F}_t^X}
\sim\left.\Q_x\right|_{\mathcal{F}_t^X}$ for all $t>0$ and the
existence of a positive $\Q_x$-martingale $D$ such that
\[
\left.\PP_x\right|_{\mathcal{F}_t^X} = D_t \cdot
\left.\Q_x\right|_{\mathcal{F}_t^X}\mbox{ for all }t\geq0.
\]
By Remark~\ref{remark} part~1, $X$ is a special semimartingale on
$(\Omega,\mathcal{F}^X,(\mathcal{F}^X_{t+}),\Q_x)$ with
decomposition
\[
X=X_0+\wB+\wX^c+z\ast(\mu^X-\wnu^X),
\]
where $\wB_t=\int_0^t \wb(X_s)\dd s$, $\wnu^X(\dd t,\dd
z)=\wK(X_t,\dd z)\dd t$ and $\wX^c$ the continuous local
martingale part with quadratic variation $\langle
\wX^c\rangle_t=\int_0^t c(X_s)\dd s$.
 A close inspection of the proof of
\cite[Theorem~2.4]{cherfilyor} reveals that
\[
D=\mathcal{E}(\phi_1(X)\cdot \wX^c+(\phi_3(X,z)-1)\ast
(\mu^X-\wnu^X)).
\]
Applying the product rule for stochastic exponentials one verifies
that
\begin{align*}
D^{-1}&=\mathcal{E}(-\phi_1(X)\cdot
X^c+(1-\phi_3(X,z))/\phi_3(X,z)\ast (\mu^X-\nu^X)),
\end{align*}
so that $D^{-1}=\mathcal{E}(H\cdot X^c+W\ast(\mu^X-\nu^X))=L$.
Since
\[
\left.\Q_x\right|_{\mathcal{F}_t^X} = D^{-1}_t \cdot
\left.\PP_x\right|_{\mathcal{F}_t^X},\mbox{ for all $t>0$},
\]
it follows that $L$ is a $\PP_x$-martingale as well as the
martingale problem for $\widetilde{\mathcal{A}}$ on $\Omega$ is
well-posed.

Now assume $\PP=\PP_\eta$ is a solution of the martingale problem
for $(\mathcal{A},\eta)$ with $\eta$ an arbitrary probability
measure on $E$. By Remark~\ref{remark} part~2,
$\Q_\lambda=\int\Q_x\lambda(\dd x)$ is the (unique) solution of
the martingale problem for $(\widetilde{\mathcal{A}},\eta)$ on
$\Omega$. Hence we can repeat the above argument with $\PP_x$ and
$\Q_x$ replaced by $\PP_\eta$ and $\Q_\eta$ to see that $L$ is a
$\PP_\eta$-martingale.
\end{proof}

\section{Affine jump-diffusions and affine processes}\label{sec:affineSDEs}
\subsection{Definitions}
We start with the definition of affine jump-diffusions and affine
processes. The former are defined from the point of view of
semimartingale theory as being jump-diffusions with \emph{affine}
differential characteristics. The latter are characterized from
the point of view of Markov process theory as having an
exponentially \emph{affine} expression for their characteristic
functions. As in the previous section we restrict ourselves to
special semimartingales.
\begin{definition}\label{def:affine}
The martingale problem for $\mathcal{A}$ given by (\ref{eq:opA})
is called an \emph{affine} martingale problem if $b$, $c$ and $K$
are affine in the sense that
\begin{equation}\label{eq:bcK}
\begin{split}
b(x)&=a^0+\sum_{i=1}^p a^i x_i\\
c(x)&=A^0+\sum_{i=1}^p A^i x_i\\
K(x,\dd z)&=K^0(\dd z)+\sum_{i=1}^p K^i(\dd z) x_i,
\end{split}
\end{equation}
for some column vectors $a^i\in\R^p$, symmetric matrices
$A^i\in\R^{p\times p}$ and (signed) measures $K^i$ on $F$
satisfying $\int(|z|^2\wedge |z|)|K^i|(\dd z)<\infty$. If the
affine martingale problem is well-posed and $\PP$ is a solution,
then the coordinate process $X$ is called an \emph{affine
jump-diffusion} on
$(\Omega,\mathcal{F}^X,(\mathcal{F}^X_{t+}),\PP)$ with state space
$E$.
\end{definition}
\begin{definition}
If the coordinate process $X$ on $\Omega=D_E[0,\infty)$ is a
Markov process with state space $E$ and transition kernel
$(\PP_x)_{x\in E}$ such that for all $u\in\ii\R^p$, $t\geq0$ we
have
\begin{align}\label{eq:char}
\E_x \exp(u^\top X_t)=\exp(\psi_0(t,u)+\psi(t,u)^\top x),\mbox{
for all }x\in E,
\end{align}
for some $\psi_0:[0,\infty)\times\ii\R^p\rightarrow\C$ and
$\psi:[0,\infty)\times\ii\R^p\rightarrow\C^p$, then
$(X,(\PP_x)_{x\in E})$ is called an \emph{affine process}. Note
that $\psi_0(t,u)$ may be altered by multiples of $2\pi\ii$. If in
addition $\psi_0$ and $\psi$ are continuously differentiable in
their first argument, it is called a \emph{regular} affine
process. In that case we put $\psi_0(0,u)=0$, so that $\psi_0$ and
$\psi$ are uniquely determined by (\ref{eq:char}).
\end{definition}

For existence of an affine jump-diffusion, restrictions need to be
imposed on the state space $E$ and the parameters
$(a^i,A^i,K^i)_{0\leq i\leq p}$ in order that $c(x)$ is a positive
semi-definite matrix and $K(x,\dd z)$ is a non-negative measure
for $x\in E$, while in addition $E$ is stochastic invariant for
$X$ (that is, $X$ does not leave the set $E$). These parameter
conditions are called \emph{admissibility conditions} and the
corresponding parameter set $(a^i,A^i,K^i)_{0\leq i\leq p}$ is
called \emph{admissible}.

Possible state spaces amongst others are the canonical state space
$\R^m_+\times\R^{p-m}$, the cone of positive semi-definite
matrices $S^p_+$ and quadratic state spaces including the
parabolic state space $\{x\in\R^p:x_1\geq \sum_{i=2}^p x_i^2\}$
and the Lorentz cone $\{x\in\R^p:x_1\geq0,x_1^2\geq \sum_{i=2}^p
x_i^2\}$, see respectively \cite{dfs03,cfmt09,part2} for the
existence and uniqueness of the associated affine jump-diffusion.
We note that the matrix-valued affine jump-diffusions are
contained in the framework of Definition~\ref{def:affine} as we
can identify symmetric matrices with vectors using the
half-vectorization operator $\vech: S^p\rightarrow\R^{p(p+1)/2}$
(the linear operator that stacks the elements from the upper
triangle of a symmetric matrix into a vector).

%


Equivalence of affine jump-diffusions and affine processes has
only been proved for the canonical state space
$\R^m_+\times\R^{p-m}$ in \cite{dfs03} with the use of the
admissibility conditions. For other state spaces this appears much
harder as the admissibility conditions become more involved, while
for arbitrary state space one has no access at all to these
conditions. One of the aims in this paper is to establish the
equivalence between affine jump-diffusions and (regular) affine
processes with an arbitrary state space under well-posedness of
the martingale problem for $\mathcal{A}$. One direction is
relatively easy and has been proved for the diffusion case in
\cite[Theorem~2.2]{fm09}. The next proposition also incorporates
jumps. The converse direction is much harder to establish and will
be proved with the least restrictions in
Section~\ref{sec:boundedexp} in Theorem~\ref{th:BsupsetU}.
\begin{prop}\label{prop:afprocisafdif}
Let $E\subset\R^p$ be closed with non-empty interior,
$E=\overline{E^\circ}$ and suppose the martingale problem for
$\mathcal{A}$ is well-posed. Let $\PP$ be a solution of the
martingale problem for $\mathcal{A}$ on $\Omega$ and $\PP_x$ for
$(\mathcal{A},\delta_x)$, $x\in E$. If $(X,(\PP_x)_{x\in E})$ is a
regular affine process, then $X$ is an affine jump-diffusion on
$(\Omega,\mathcal{F}^X,(\mathcal{F}^X_{t+}),\PP)$ with state space
$E$, say with differential characteristics $(b(X),c(X),K(X,\dd
z))$ given by (\ref{eq:bcK}).
Moreover, for all $u\in\ii\R^p$ it holds that
$(\psi_0(\cdot,u),\psi(\cdot,u))$ characterized by (\ref{eq:char})
and $\psi_0(0,u)=0$, solves the system of generalized Riccati
equations
\begin{align}\label{eq:riccati}
\dot{\psi_i}=R_i(\psi),\quad \psi_i(0)=u_i,\quad i=0,\ldots,p,
\end{align}
with
\begin{align}\label{eq:Ri}
R_i(y)=y^\top a^i +\half y^\top A^i y +\int (e^{y^\top z}-1-y^\top
z)K^i(\dd z),
\end{align}
where we write $u_0=0$.
\end{prop}
\begin{proof}
Fix $T>0$ and $u\in\ii\R^p$. By the Markov property, it holds
$\PP$-almost surely that
\begin{align*}
\E \exp(u^\top X_T|\mathcal{F}^X_t)&=\E_{X_t} \exp(u^\top
X_{T-t})\\&=\exp(\psi_0(T-t,u)+\psi(T-t,u)^\top X_t)=:f(t,X_t),
\end{align*}
for all $t\leq T$.
For convenience in the next display we write $\psi$ and
$\dot{\psi}$ instead of $\psi(T-t,u)$ and $\dot{\psi}(T-t,u)$. By
Remark~\ref{remark} part~1, $X$ is a special jump-diffusion and
admits the decomposition (\ref{eq:decompx}). It\^{o}'s formula
gives
\begin{equation}\label{eq:ftXt}
\begin{split}
\frac{\dd f(t,X_t)}{f(t,X_{t-})}&=(-\dot{\psi_0}-\dot{\psi}^\top
X_t)\dd t +\psi^\top\dd X_t +\half \psi^\top c(X_t)\psi\dd t
\\&+\int_{z\in F}(e^{\psi^\top
z}-1-\psi^\top z)\mu^X(\dd t,\dd z)\\
&=\psi^\top \dd X_t^c +\int_{z\in F}(e^{\psi^\top
z}-1)(\mu^X-\nu^X)(\dd t,\dd z)+I(t,X_t)\dd t,
\end{split}
\end{equation}
with
\[
I(t,x)=-\dot{\psi_0}-\dot{\psi}^\top x+\psi^\top
b(x)+\half\psi^\top c(x)\psi+\int(e^{\psi^\top z}-1-\psi^\top z
)K(x,\dd z),
\]
and all expression are well-defined as $f$ is bounded, see
\cite[Theorem~II.2.42]{JacShir}. Since $f(t,X_t)$ is a
$\PP$-martingale, it follows that $\int_0^t I(s,X_s)\dd s=0$,
$\PP$-a.s. Right-continuity of $I(t,X_t)$ yields that $I(t,X_t)=0$
for all $t\geq0$, $\PP$-a.s. In particular $I(0,X_0)=0$,
$\PP$-a.s. Choosing $\PP=\PP_x$ for $x\in E$, we obtain $I(0,x)=0$
for all $x\in E$, i.e.\
\begin{equation}\label{eq:dotpsi}
\begin{split}
\dot{\psi_0}(T,u)+\dot{\psi}(T,u)^\top x&=\psi(T,u)^\top
b(x)+\half\psi(T,u)^\top c(x)\psi(T,u)\\&+\int(e^{\psi(T,u)^\top
z}-1-\psi(T,u)^\top z )K(x,\dd z).
\end{split}
\end{equation}
This holds for all $T\geq0$, $u\in\ii\R^p$. In particular it holds
for $T=0$. We have $\psi(0, u)= u$ for $u\in\R^p$. Write $u=\ii y$
for $y\in\R^p$, then we get
\begin{align*}
\dot{\psi_0}(0,\ii y)+\dot{\psi}(0, \ii y)^\top x= \ii y^\top
b(x)-y^\top c(x)y+\int(e^{ \ii y^\top z}-1- \ii y^\top z )K(x,\dd
z),
\end{align*}
for all $y\in\R^p$. Differentiating the left- and right-hand side
with respect to $y_i$ in $y_i=0$ and putting $y_k=0$ for $k\neq i$
gives that $b_i(x)$ is affine for all $i\leq p$. Dividing the
left- and right-hand side by $y_i y_j$ for $i,j\leq p$, putting
$y_k=0$ for $k\neq i,j$ and letting $y_i\rightarrow\infty$,
$y_j\rightarrow\infty$, we deduce that $c_{ij}(x)$ is affine.
Hence $c(x)$ is affine and also $\int(e^{ \ii y^\top z}-1- \ii
y^\top z )K(x,\dd z)$ is affine in $x$ for all $y\in\R^p$. To show
that $K(x,\dd z)$ is affine in $x$, we fix $k\in E^\circ$
arbitrary and take $\varepsilon>0$ such that
\[
\{x\in\R^p:k_i\leq x_i\leq k_i+\varepsilon\mbox{ for all }i
\}\subset E.
\]
Define
\begin{align*}
K^0(\dd z)&=K(k,\dd z)-\sum_{i=1}^p (K(k+\varepsilon e_i,\dd
z)-K(k,\dd z))k_i/\varepsilon\\ K^i(\dd z)&=(K(k+\varepsilon
e_i,\dd z)-K(k,\dd z))/\varepsilon,\qquad\mbox{ for
$i=1,\ldots,p$}.
\end{align*}
Then it follows that
\[
\int(e^{ u^\top z}-1- u^\top z )K(x,\dd z)=\int(e^{ u^\top z}-1-
u^\top z )(K^0(\dd z)+\sum_{i=1}^p K^i(\dd z)x_i),
\]
for all $u\in\ii\R^p$, $x\in E$, since the left-hand side is
affine and is uniquely determined by the values at $x=k$ and
$x=k+\varepsilon e_i$, $i=1,\ldots,p$. Equality of the left- and
right-hand side at these points follows from the identity
\begin{align*}
K^0(\dd z)+\sum_{i=1}^p K^i(\dd z)x_i &= K(k,\dd
z)(1+\sum_{i=1}^p(k_i-x_i)/\varepsilon)\\&+\sum_{i=1}^p
K(k+\varepsilon e_i,\dd z)(x_i-k_i)/\varepsilon.
\end{align*}
Note that the right-hand side is a non-negative measure for $x\in
B_k$, where $B_k$ is given by
\[
B_k:=\{x\in\R^p:k_i\leq x_i\leq k_i+\varepsilon/p \mbox{ for all
}i \}.
\]
By uniqueness of the L\'{e}vy triplet (see
\cite[Lemma~II.2.44]{JacShir}), this yields that $K(x,\dd
z)=K^0(\dd z)+\sum_{i=1}^p K^i(\dd z)x_i$ for $x\in B_k$. Since
$k\in E^\circ$ is chosen arbitrarily, we have an affine expression
for $K(x,\dd z)$ on a neighborhood of each $x\in E^\circ$. From
this it follows that $K(x,\dd z)$ is affine on the whole of
$E=\overline{E^\circ}$. Hence $X$ is an affine jump-diffusion. Let
the differential characteristics $(b(X),c(X),K(X,\dd z))$ be given
by (\ref{eq:bcK}). Plugging these into (\ref{eq:dotpsi}) and
separating first order terms in $x$ gives (\ref{eq:riccati}).
\end{proof}

\subsection{The affine transform formula}

The expression (\ref{eq:char}) where $(\psi_0,\psi)$ solve the
system of Riccati equations (\ref{eq:riccati}), is called the
\emph{affine transform formula}. In the previous subsection we
obtained this formula for the characteristic function of an affine
process, with a general state space. This subsection is devoted to
the validity of the affine transform formula for affine
\emph{jump-diffusions} with a general state space, for arbitrary
parameters $u\in\C^p$. The key step is the following proposition
which is a direct application of Theorem~\ref{th:expmart}.
\begin{prop}\label{prop:Lismart}
Suppose the affine martingale problem for $\mathcal{A}$ given by
(\ref{eq:opA}) and (\ref{eq:bcK}) is well-posed. Let
$h:E\rightarrow\R^p$ and $w:E\times F\rightarrow (-1,\infty)$ be
measurable, write $H_t=h(X_t)$, $W(t,z)=w(X_t,z)$ and let $\PP$ be
a solution of the martingale problem for $\mathcal{A}$ on
$\Omega$, which yields the decomposition (\ref{eq:decompx}) for
$X$. Then
\[
L=\mathcal{E}(H\cdot X^c+W\ast(\mu^X-\nu^X))
\]
is an $((\mathcal{F}^X_{t+}),\PP)$-martingale under the additional
assumptions
\begin{enumerate}
\item $h$ is bounded and continuous,
\item $x\mapsto\int |z| w(x,z) |K^i|(\dd z)$ is continuous and
finite
\item $x\mapsto \int (|z|^2\wedge |z|)(w(x,z)+1)|K^i|(\dd z)$ is
continuous and finite,
\item $\int |z|^2(w(x,z)+1)|K^i|(\dd z)|x_i|\leq C(1+|x|^2)$, for some $C>0$, all $x\in E$,
\item $x\mapsto \int(w(x,z)-\log(w(x,z)+1))|K^i|(\dd z)$ is bounded on
compacta,
\item $\int |z|^q\log|z|(w(x,z)+1)|K^i|(\dd z)|x_i|\leq C(1+|x|^q)$, for some
$C>0$, $q>0$, all $x\in E$,
\end{enumerate}
for all $i=0,\ldots,p$, where we write $x_0:=1$. Furthermore, the
martingale problem for $\widetilde{\mathcal{A}}$ given by
(\ref{eq:opwA}) and (\ref{eq:wbwK}) is well-posed.
\end{prop}
\begin{proof}
This is a reformulation of Theorem~\ref{th:expmart} for the affine
martingale problem. One has to check conditions
(\ref{eq:condbcK}), (\ref{eq:condgrowK}), (\ref{eq:condcontinu}),
 (\ref{eq:condmoregrowK}), (\ref{eq:growth}) and
(\ref{eq:condbcK2}), which is left to the reader.
\end{proof}
Using the above proposition we validate the affine transform
formula under existence of the solutions to the Riccati equations
in the following theorem, which is the first main result of the
paper. The imposed assumptions are in the same spirit as
\cite[Theorem~5.1]{kallkarb08}.
\begin{theorem}\label{th:ricexists}
Let $X$ be an affine jump-diffusion  with differential
characteristics $(b(X),c(X),K(X,\dd z))$ given by (\ref{eq:bcK})
on $(D_E(0,\infty],\mathcal{F}^X,(\mathcal{F}^X_{t+}),\PP)$. Let
$u\in\R^p$,  $T>0$ and suppose $\psi_0\in C^1([0,T],\R)$ and
$\psi\in C^1([0,T],\R^p)$ solve the system of generalized Riccati
equations given by (\ref{eq:riccati}) (with $u_0:=0$). Under the
assumptions
\begin{enumerate}
\item $\sup_{t\leq T}\int |z|^2 e^{\psi(t)^\top z} |K^i|(\dd z)<\infty$, for $i=0,\ldots,p$,
\item $t\mapsto\int_{\{|z|>1\}} |z|e^{\psi(t)^\top z}|K^i|(\dd z)$
is continuous for all $i=0,\ldots,p$,
\item $\E\exp(\psi(T)^\top X_0)<\infty$,
\end{enumerate}
it holds that
\[
\E(\exp(u^\top
X_T)|\mathcal{F}^X_{t+})=\exp(\psi_0(T-t)+\psi(T-t)^\top
X_t),\mbox{ for all $t\leq T$}.
\]
\end{theorem}
\begin{proof}
To prove Theorem~\ref{th:ricexists} it suffices to show that
$f(t,X_t)$ given by $f(t,X_t)=\exp(\psi_0(T-t)+\psi(T-t)^\top X_t$
is an $((\mathcal{F}^X_{t+}),\PP)$-martingale on $[0,T]$, since
$f(T,X_T)=\exp(u^\top X_T)$ in view of the initial condition of
$(\psi_0,\psi)$. We restrict time to $[0,T]$. We have
(\ref{eq:ftXt}) with $I(t,X_t)=0$, since $(\psi_0,\psi)$ satisfy
(\ref{eq:riccati}). Hence $M_t:=f(t,X_t)$ satisfies
\[
M=M_0\,\mathcal{E}(\psi(T-t)\cdot X^c+(e^{\psi(T-t)^\top z}-1)\ast
(\mu^X-\nu^X)).
\]
Write $Y=(X_t,t)$ and note $Y$ is an affine jump-diffusion with
state space $E\times[0,T]$.
We define $h(x,t)=\psi(T-t)$ and $w(x,t,z)=e^{\psi(T-t)^\top
z}-1$. Write $H=h(Y)$, $W(t,z)=w(Y_t,z)$, then we deduce that
\[
L:=\mathcal{E}(H\cdot Y^c+W\ast(\mu^{Y}-\nu^{Y}))
\]
is an $((\mathcal{F}^X_{t+}),\PP)$-martingale by applying
Proposition~\ref{prop:Lismart} to the affine jump-diffusion $Y$.
One easily verifies that the assumptions in that proposition are
met. Since $M_t=M_0L_t$ and $\E M_0<\infty$, it follows that $M$
is an $((\mathcal{F}^X_{t+}),\PP)$-martingale on $[0,T]$, as we
needed to show.
\end{proof}

Theorem~\ref{th:maintheorem} below is our second main result. We
establish the full-range of validity of the affine transform
formula under all finite exponential moments for the tails of the
jump-measures $K^i$, for affine jump-diffusion with a general
closed convex state space, extending \cite[Theorem~3.3]{fm09}. The
proof is divided over the next two sections. We use the results
and notation from \cite[Lemma~2.3 and Lemma~A.2]{fm09}, which we
state as a proposition for ease of reference.
\begin{prop}\label{prop:filipmayer}
Suppose
\begin{align}\label{eq:exponentialmom}
\int_{\{|z|>1\}} e^{k^\top z}|K^i|(\dd z)<\infty,\mbox{ for all
}k\in\R^p, i=0,\ldots,p.
\end{align}
Let $\K$ be a placeholder for either $\R$ or $\C$. It holds that
\begin{enumerate}[(i)]
\item\label{item:fm1} For all $ u\in\K^p$ there exists an ``explosion-time''
$t_\infty( u)>0$ such that there exists a unique solution
$(\psi_0(\cdot, u),\psi(\cdot, u)):[0,t_\infty( u))\rightarrow
\K\times\K^p$ to the system of Riccati equations
(\ref{eq:riccati}), where either $t_\infty( u)=\infty$ or
$\lim_{t\uparrow t_\infty( u)}\|\psi(t, u)\|=\infty$. In
particular $t_\infty(0)=\infty$.
\item\label{item:fm2}  The set
\[
D_\K:=\{ (t,u)\in [0,\infty)\times\K^p:t<t_\infty( u)\},
\]
is open in $[0,\infty)\times \K^p$ and the $\psi_i$ are analytic
on $D_\K$. In addition, for all $t\geq0$
\[
D_\K(t):=\{ u\in\C^p:(t,u)\in D_\K\}
\]
is an open neighborhood of $0$ and $D_\K(t_2)\subset D_\K(t_1)$
for $0\leq t_1\leq t_2$.
\item\label{item:fm3} If $O\subset\R^p$ is an open set and $\nu$ is a bounded measure such that we have $\int\exp( u^\top
x)\dd\nu(x)<\infty$ for all $ u\in O$, then $ u\mapsto\int\exp(
u^\top x)\dd\nu(x)$ is analytic on the open strip
\[
S(O):=\{z\in\C^p:\Re z\in O\}.
\]
\end{enumerate}
\end{prop}
\begin{theorem}\label{th:maintheorem}
Suppose $E\subset\R^p$ is closed convex with non-empty interior
and let $X$ be an affine jump-diffusion on
$(D_E(0,\infty],\mathcal{F}^X,(\mathcal{F}^X_{t+}),\PP)$ with
differential characteristics $(b(X),c(X),K(X,\dd z))$ given by
(\ref{eq:bcK}). Assume (\ref{eq:exponentialmom}) and let the
notation of Proposition~\ref{prop:filipmayer} be in force. Then
for $t>0$ it holds that
\begin{enumerate}[(i)]

\item\label{twee}  $D_\R(t)=M(t)$, where
\[
M(t)=\{ u\in\R^p:\E_x(\exp( u^\top X_t))<\infty\mbox{ for all
}x\in E \}.
\]
\item\label{een} $S(D_\R(t))\subset D_\C(t)$.
\item\label{vier}  The affine transform formula (\ref{eq:char}) holds for all $ u\in
S(D_\R(t))$.
\item\label{drie}  $D_\R(t)$ and $D_\R$ are
convex sets.
\item\label{vijf} $M(t)\subset M(s)$ for $0\leq s\leq t$.
\end{enumerate}
\end{theorem}
\begin{proof}
Theorem~\ref{th:ricexists} yields $D_\R(t)\subset M(t)$. The proof
of $D_\R(t)\supset M(t)$ is the content of
Section~\ref{sec:realexp}, while Section~\ref{sec:complexexp} is
devoted to the proof of (\ref{een}) and (\ref{vier}).
Assertions~(\ref{drie}) and (\ref{vijf}) follow from (\ref{twee}).
\end{proof}
\section{Full range of validity for real exponentials}\label{sec:realexp}

Let $T>0$. In order to prove $M(T)\subset D_\R(T)$, we show that
$\psi(T, u)$ explodes when $ u\in D_\R(T)$ approaches the boundary
$\partial (D_\R(T))$. This is not immediate as the following
example demonstrates.
\begin{example}
Consider the Riccati equation $\dot{x}=x^2$. Its solution $x$ with
initial condition $ u\in\C$ is given by $x(t, u)= u/(1- u t)$ and
we have $t_\infty( u)= u^{-1}$ for $ u\in\R_{>0}$ and $t_\infty(
u)=\infty$ otherwise. Hence $D_\C(T)=\{ u\in\C: u\not\in
[T^{-1},\infty) \}$ and $\partial D_\C(T)=[T^{-1},\infty)$.
Obviously $x(T, u)$ does not explode if $ u\in D_\C(T)$ tends to $
u_0\in(T^{-1},\infty)$. If we take real and imaginary part, then
we obtain a $2$-dimensional system of Riccati equations given by
\begin{align*}
\dot{x_1}&= x_1^2 - x_2^2\\
\dot{x_2}&=2x_1x_2.
\end{align*}
In this case $D_\R(T)=\{ u\in\R^2: u\not\in [T^{-1},\infty) \}$
and again $x(t, u)$ does not explode if $ u\in D_\R(T)$ tends to $
u_0\in(T^{-1},\infty)$. Note that the Riccati equations are of the
form (\ref{eq:riccati}) (excluding the equation for $\psi_0$) with
\[
A^1=
  \begin{pmatrix}
    1 & 0 \\
    0 & -1
  \end{pmatrix},\,
A^2=
  \begin{pmatrix}
    0 & 1 \\
    1 & 0
  \end{pmatrix},\,a=0.
\]
However, they are not related to an affine diffusion where the
state space has non-empty interior. Indeed, the corresponding
diffusion matrix would be
\[
c(x)=\begin{pmatrix}
    x_1 & x_2 \\
    x_2 & -x_1
  \end{pmatrix},
\]
which is positive semi-definite if and only if $x=0$.
\end{example}
\medskip

In Lemma~\ref{prop:phipluspsi} below we derive a formula that
relates solutions to Riccati equations to the expectation of the
corresponding affine diffusion. This will turn out to be most
useful in Proposition~\ref{prop:Expinfty} to derive that
$M(T)\subset D_\R(T)$, which proves
Theorem~\ref{th:maintheorem}~(\ref{twee}).
\begin{prop}\label{prop:phipluspsi}
Consider the situation of Theorem~\ref{th:maintheorem}. Define the
non-negative function $k:E\times\R^p\rightarrow\R$ by
\begin{equation}\label{eq:kxy}
k(x,y)=\half y^\top c(x) y + \int (e^{y^\top z}-1-y^\top z)K(x,\dd
z),
\end{equation}
for $x\in E$ and $y\in\R^p$. Then for all $x\in E$, $u\in\R^p$,
$t<t_\infty(u)$ it holds that
\begin{equation}\label{eq:mooieformule}
\begin{split}
\psi_0(t,u)+\psi(t,u)^\top x &=  u^\top \E_x X_t +\int_0^t k(\E_x
X_{t-s},\psi(s,u))\dd s,
\end{split}
\end{equation}
and $\E_x X_t$ solves the linear ODE
\begin{align}\label{eq:linODE}
\dot x = b(x),\quad x(0)=x.
\end{align}
\end{prop}
\begin{proof}
Fix $u\in\R^p$ and write $\psi(\cdot)$ instead of $\psi(\cdot,u)$.
We can write the ODE for $(\psi_0,\psi)$ as an inhomogeneous
linear ODE, namely
\[
  \begin{pmatrix}
    \dot{\psi_0} \\
    \dot{\psi}
  \end{pmatrix}=
  A \begin{pmatrix}
   \psi_0 \\
   \psi
  \end{pmatrix}+g,\,\mbox{
with } A=\begin{pmatrix}
    0 & {a^0}^\top \\
    0 & a^\top
  \end{pmatrix},
\]
where we write $a$ for the $(p\times p)$-matrix with columns
$a^i$, $i=1,\ldots,p$ and $g=(g_0,g_1,\ldots,g_p)$ is the function
given by
\[
 g_i=
   \half\psi^\top A^{i} \psi+\int (e^{\psi^\top
z}-1-\psi^\top z)K^{i}(\dd z),\quad i=0,\ldots,p.
\]
By an application of a variation of constants, the solution can be
written as
\[
\begin{pmatrix}
   \psi_0(t) \\
   \psi(t)
  \end{pmatrix}=e^{At}\begin{pmatrix}
   0 \\
    u
  \end{pmatrix} + \int_0^t e^{A(t-s)}g(s)\dd s,
\]
which yields
\begin{equation}
\begin{split}\label{eq:phipluspsi} \psi_0(t)+\psi(t)^\top x
&=\begin{pmatrix}
   \psi_0(t) \\
   \psi(t)
  \end{pmatrix}^\top
  \begin{pmatrix}
    1 \\
    x
  \end{pmatrix}\\&=\begin{pmatrix}
   0 &
    u^\top
  \end{pmatrix}e^{A^\top t} \begin{pmatrix}
    1 \\
    x
  \end{pmatrix} + \int_0^t g(s)^\top e^{A^\top(t-s)}\begin{pmatrix}
    1 \\
    x
  \end{pmatrix}\dd s.
  \end{split}
\end{equation}
Write $f(t,x)$ for the solution to the linear
ODE~(\ref{eq:linODE}) with $f(0,x)=x$. Then we have
\[
\begin{pmatrix}
    y(t) \\
    z(t)
  \end{pmatrix}:=e^{A^\top t} \begin{pmatrix}
    1 \\
    x
  \end{pmatrix}= \begin{pmatrix}
    1 \\
    f(t,x)
  \end{pmatrix}.
\]
Indeed, since
\[
\begin{pmatrix}
    \dot{y} \\
    \dot{z}
  \end{pmatrix}=A^\top \begin{pmatrix}
    y \\
    z
  \end{pmatrix}=
  \begin{pmatrix}
    0  \\
    a^0 y + a z
  \end{pmatrix},
\]
it holds that $y=1$ and $\dot{z}=a z+a^0=b(z)$ with $z(0)=x$,
whence $z(t)=f(t,x)$. Noting that
\[
g^\top
\begin{pmatrix}
    1 \\
    x
  \end{pmatrix}=\half \psi^\top c(x)\psi+\int (e^{\psi^\top z}-1-\psi^\top z)K(x,\dd z),\,\mbox{ for all $x\in E$},
  \]
and $\E_x X_t\in E$ for all $x\in E$, $t\geq0$, by convexity of
$E$,  we obtain (\ref{eq:mooieformule}) from (\ref{eq:phipluspsi})
after we have shown that $\E_x X_t = f(t,x)$. The latter follows
from Lemma~\ref{lem:sup}, as it yields
\[
\E_x X_t=x+\E_x \int_0^t (a^0+ a X_s)\dd s= \int_0^t (a^0+ a \E_x
X_s)\dd s.
\]
\end{proof}
\medskip

In the following we make use of the fact that for $c_n\in\R^p$ it
holds that
\begin{align}\label{eq:limcn}
\lim_{n\rightarrow\infty}\|c_n\|=\infty \Rightarrow \exists
x\in\{-1,1\}^p,\varepsilon>0:\limsup_{n\rightarrow\infty}
\inf_{y\in B(x,\varepsilon)} c_n^\top y=\infty.
\end{align}
Indeed, if $\lim_{n\rightarrow\infty}\|c_n\|=\infty $, then there
exists a subsequence $c_{n_k}$ such that all components
$c_{n_k,i}$ are convergent in $[-\infty,\infty]$. In addition, one
of them converges to either $+\infty$ or $-\infty$. Define
$x\in\R^p$ by taking $x_i=-1$ if $c_{n_k,i}\rightarrow-\infty$ and
$x_i=1$ otherwise. Then obviously for $y\in B(x,\varepsilon)$ with
$0<\varepsilon<1$ we have
\[
\inf_{y\in B(x,\varepsilon)}c_{n_k}^\top
y\rightarrow\infty,\,\mbox{ as }k\rightarrow\infty.
\]
\begin{lemma}\label{lem:Tistinfty}
Consider the situation of Theorem~\ref{th:maintheorem}. Let $
u\in\R^p$ and suppose $T:=t_\infty( u)<\infty$. Then there exists
$x\in E$ such that $\E_x\exp( u^\top X_T)=\infty$.
\end{lemma}
\begin{proof}
Without loss of generality we may assume that $\{-1,1\}^p\subset
E^\circ$ (thus by convexity also $0\in E^\circ$).
Since $\|\psi(t, u)\|\rightarrow\infty$ for $t\uparrow T$ and in
view of (\ref{eq:limcn}), there exists a ball
$B:=B(x_0,\varepsilon)\subset E$ (with $x_0\in\{-1,1\}^p\subset
E^\circ$, $\varepsilon>0$) and a sequence $t_n\uparrow T$ such
that
\[
\inf_{y\in B} \psi(t_n, u)^\top y \rightarrow\infty\mbox{ as
$n\rightarrow\infty$}.
\]
Moreover, it holds that $\psi_0(t, u)\geq u^\top \E_0(X_t)$ for
$t<T$ by Proposition~\ref{prop:phipluspsi}. In particular we have
$\liminf_{t\uparrow T}\psi_0(t, u)>-\infty$. Hence
\[
\lim_{n\rightarrow\infty}\inf_{y\in B} (\psi_0(t_n, u)+\psi(t_n,
u)^\top y )=\infty.
\]
By right-continuity of $X$, it follows that
\[
\lim_{n\rightarrow\infty}(\psi_0(t_n, u)+\psi(t_n, u)^\top
X_{T-t_n})=\infty,\mbox{ $\PP_{x_0}$-a.s.}
\]
The Markov property and Theorem~\ref{th:ricexists} give
\begin{align*}
\E_x\exp( u^\top X_T)&=\E_x\left(\E_{X_{T-t}}\exp( u^\top
X_{t})\right)=\E_x \exp(\psi_0(t, u)+\psi(t, u)^\top X_{T-t}),
\end{align*}
for $0\leq t< T$, $x\in E$. Applying the previous together with
Fatou's Lemma we get
\begin{align*}
\E_{x_0}\exp( u^\top
X_T)&=\liminf_{n\rightarrow\infty}\E_{x_0}\exp( u^\top
X_T)\\&=\liminf_{n\rightarrow\infty}\E_{x_0} \exp(\psi_0(t_n,
u)+\psi(t_n, u)^\top X_{T-t_n})\\&\geq \E_{x_0}
\liminf_{n\rightarrow\infty}\exp(\psi_0(t_n, u)+\psi(t_n, u)^\top
X_{T-t_n})=\infty.
\end{align*}
\end{proof}
\begin{prop}\label{prop:Expinfty}
Consider the situation of Theorem~\ref{th:maintheorem}. Let
$T\geq0$. Then $M(T)= D_\R(T)$ and (\ref{eq:char}) holds for $u\in
M(T)$, $t\leq T$.
\end{prop}
\begin{proof}
In view of Theorem~\ref{th:ricexists} it is sufficient to prove
$M(T)\subset D_\R(T)$. Without loss of generality we may assume
that $\{-1,1\}^p\subset E^\circ$.  Let $ u\in\R^p$ and suppose
$t_\infty( u)<\infty$. We need to show that for all $T\geq
t_\infty( u)$ there exists $x\in E$ such that $\E_x\exp( u^\top
X_T)=\infty$. Lemma~\ref{lem:Tistinfty} gives the result for
$T=t_\infty( u)$. Therefore, let $T>t_\infty( u)$. Arguing by
contradiction, assume $\E_x\exp( u^\top X_T)<\infty$ for all $x\in
E$. Then by Jensen's inequality we have
\begin{align}\label{eq:Jensen}
\E_x\exp(\lambda u^\top X_T)\leq(\E_x\exp( u^\top X_T))^\lambda
\leq 1+\E_x\exp( u^\top X_T)<\infty,
\end{align}
for all $0\leq\lambda\leq 1$, $x\in E$. Let
$\lambda^*=\inf\{\lambda\geq 0:\lambda u\not\in D_\C(T)\}$. Note
that $0<\lambda^*\leq 1$ and $\lambda^* u\not\in D_\C(T)$, since $
u\not\in D_\C(T)$ and $D_\C(T)$ is an open neighborhood of $0$.
Considering $\lambda^* u$ instead of $ u$, we may assume without
loss of generality that $\lambda^*=1$. In the following, we let $
u_n=\lambda_n u$, for arbitrary $\lambda_n\in[0,1)$ such that
$\lambda_n\uparrow 1$ as $n\rightarrow\infty$, so that $ u_n\in
D_\C(T)$ and $ u_n\rightarrow u$. We divide the proof into a
couple of steps.

\emph{Step 1.} If for some $t\leq T$ and $x\in E$ we have
\begin{align}\label{eq:phipsiinfty}
 \lim_{n\rightarrow\infty} (\psi_0(t, u_n)
+\psi(t, u_n)^\top x)=\infty,
\end{align}
then $\limsup_{n\rightarrow\infty}\|\psi(t, u_n)\|=\infty$. To
prove this, suppose (\ref{eq:phipsiinfty}) holds for some $t\leq
T$, but $\limsup_{n\rightarrow\infty}\|\psi(t, u_n)\|<\infty$.
Then $\lim_{n\rightarrow\infty}\psi_0(t, u_n)=\infty$ and
(\ref{eq:phipsiinfty}) holds for all $x$. Since $ u_n\in
D_\C(T)\subset D_\C(t)$,  the Markov property and
Theorem~\ref{th:ricexists} give
\begin{align*}
\E_x \exp( u_n^\top X_T)&=\E_x\left(\E_{X_{T-t}}\exp( u_n^\top
X_{t})\right) \\&=\E_x \exp(\psi_0(t, u_n)+\psi(t, u_n)^\top
X_{T-t}).
\end{align*}
Fatou's Lemma yields
\begin{align*}
\infty&=\E_x \liminf_{n\rightarrow\infty}\exp(\psi_0(t,
u_n)+\psi(t, u_n)^\top
X_{T-t})\\&\leq\liminf_{n\rightarrow\infty}\E_x \exp(\psi_0(t,
u_n)+\psi(t, u_n)^\top X_{T-t})=\liminf_{n\rightarrow\infty}\E_x
\exp( u_n^\top X_T),
\end{align*}
which contradicts (\ref{eq:Jensen}) as $ u_n=\lambda_n u$ with
$0\leq \lambda_n<1$.

\emph{Step 2.} It holds that
\begin{align}\label{eq:psiinfty}
\limsup_{n\rightarrow\infty}\|\psi(t_\infty( u), u_n)\|=\infty.
\end{align}
Indeed, since $ u_n\in D_\C(T)\subset D_\C(t_\infty( u))$, Fatou's
Lemma together with Theorem~\ref{th:ricexists} gives
\begin{align*}
\E_x \exp( u^\top X_{t_\infty( u)})&\leq
\liminf_{n\rightarrow\infty}\E_x\exp( u_n^\top X_{t_\infty(
u)})\\&=\liminf_{n\rightarrow\infty}\exp(\psi_0(t_\infty( u), u_n)
+\psi(t_\infty( u), u_n)^\top x),
\end{align*}
for all $x\in E$. In view of Lemma~\ref{lem:Tistinfty} there
exists an $x_0\in E$ such that we have $\E_{x_0} \exp( u^\top
X_{t_\infty( u)})=\infty$, whence
\[
\psi_0(t_\infty( u), u_n) +\psi(t_\infty( u), u_n)^\top
x_0\rightarrow\infty,\mbox{ as }n\rightarrow\infty.
\]
Step~1 yields (\ref{eq:psiinfty}).

\emph{Step 3.} It holds that
$\limsup_{n\rightarrow\infty}\|\psi(T, u_{n})\|=\infty$. To prove
this, we show that there exists $\varepsilon>0$ such that if
$\limsup_{n\rightarrow\infty}\|\psi(t_0, u_n)\|=\infty$ for some
$t_0\in [t_\infty( u),T]$, then
$\limsup_{n\rightarrow\infty}\|\psi(t_1, u_{n})\|=\infty$ for
$t_1=T\wedge(t_0+\varepsilon)$. By Step~2 and an iteration of the
above implication, it follows that
$\limsup_{n\rightarrow\infty}\|\psi(T, u_{n})\|=\infty$.

Write $f(t,x)$ for the solution to the linear
ODE~(\ref{eq:linODE}) with $f(0,x)=x$. By continuity of $f$ and
the assumption $\{-1,1\}^p\subset E^\circ$, there exists
$\varepsilon>0$ such that $f(-t,x)\in E$ for all $x\in\{-1,1\}^p$,
$0\leq t\leq\varepsilon$. Let $t_0\in [t_\infty( u),T]$ and
$t_1=T\wedge(t_0+\varepsilon)$. Suppose
$\limsup_{n\rightarrow\infty}\|\psi(t_0, u_n)\|=\infty$. Then in
view of (\ref{eq:limcn}), there exist $x\in\{-1,1\}^p$ and a
subsequence of $ u_n$ (also denoted by $ u_n$) such that
\[
\lim_{n\rightarrow\infty}\psi(t_0, u_n)^\top x =\infty.
\]
As in the proof of Lemma~\ref{lem:Tistinfty} we have
$\liminf_{n\rightarrow\infty} \psi_0(t_0, u_n)>-\infty$. Hence
\begin{align}\label{eq:limphipluspsi}
\lim_{n\rightarrow\infty}(\psi_0(t_0, u_n)+\psi(t_0, u_n)^\top x)
=\infty.
\end{align}
Since $t_0-t_1\geq-\varepsilon$, we have $y:=f(t_0-t_1,x)\in E$
and by the semi-group property of the flow it holds that
\[
\E_y X_{t_1-s}=f(t_1-s,f(t_0-t_1,x))=f(t_0-s,x)=\E_x
X_{t_0-s},\mbox{ for $s\leq t_0$}.
\]
Let $k$ be the non-negative function given by (\ref{eq:kxy}). It
follows from Proposition~\ref{prop:phipluspsi} that
\begin{align*}
\psi_0(t_1, u_n)+\psi(t_1, u_n)^\top y&= u^\top_n \E_y X_{t_1}
+\int_0^{t_1} k(\E_y X_{t_1-s},\psi(s, u_n))\dd s\\
&\geq  u^\top_n \E_y X_{t_1} +\int_0^{t_0} k(\E_y X_{t_1-s},\psi(s, u_n))\dd s\\
&=  u^\top_n \E_y X_{t_1} +\int_0^{t_0} k(\E_x X_{t_0-s},\psi(s, u_n))\dd s\\
&=  u^\top_n (\E_y X_{t_1} - \E_x X_{t_0})+ \psi_0(t_0,
u_n)+\psi(t_0, u_n)^\top x,
\end{align*}
which tends to infinity as $n\rightarrow\infty$. Step~1 yields
$\limsup_{n\rightarrow\infty}\|\psi(t_1, u_n)\|=\infty$.

\emph{Step 4.} We are now able to conclude the proof. By Step~3
and (\ref{eq:limphipluspsi}) with $t_0=T$, there is an
$x\in\{-1,1\}^p$ and a subsequence of $ u_n$ (also denoted by $
u_n$) such that
\[
\lim_{n\rightarrow\infty}(\psi_0(T, u_n)+\psi(T, u_n)^\top
x)=\infty.
\]
From (\ref{eq:Jensen}) and Theorem~\ref{th:ricexists} we obtain
\begin{align*}
1+\E_x \exp( u^\top X_T)&\geq \E_x \exp( u_n^\top X_T)=
\exp(\psi_0(T, u_n)+\psi(T, u_n)^\top x),
\end{align*}
for all $n$. The right-hand side tends to infinity, whence $\E_x
\exp( u^\top X_T)=\infty$, contrary to the assumption.
\end{proof}

\section{Extending the validity to complex exponentials}\label{sec:complexexp}

To show that $S(M(T))\subset D_\C(T)$ we need continuity of
$x\mapsto \E_x \exp(u^\top X_T)$. We prove this first in the next
lemma, together with some additional results needed in
Section~\ref{sec:boundedexp}.
\begin{lemma}\label{lem:continx}
Let $X$ be an affine jump-diffusion on
$(D_E(0,\infty],\mathcal{F}^X,(\mathcal{F}^X_{t+}),\PP)$ with
differential characteristics $(b(X),c(X),K(X,\dd z))$ given by
(\ref{eq:bcK}). Assume
\begin{align}\label{eq:z2Kfinite}
 \int |z|^2 |K^i|(\dd z)<\infty,\mbox{ for
all }i=0,\ldots,p.
\end{align}
and let $u\in\C^p$ be such that $\sup_{x\in E} \Re u^\top
x<\infty$. Suppose there exists functions $\Psi_0:[0,T)\mapsto
\C$, $\psi:[0,T)\mapsto \C^p$ such that $\Psi_0(t)\neq 0$ for
$t<T$ and
\[
\E_x \exp(u^\top X_t)=\Psi_0(t)\exp(\psi(t)^\top x)\mbox{ for all
$x\in E$, $t<T$}.
\]
Then there exists a function $\psi_0$ such that
$\Psi_0(t)=\exp(\psi_0(t))$ and $(\psi_0,\psi)$ solve the system
of generalized Riccati equations (\ref{eq:riccati}) on $[0,T)$.
Moreover, $x\mapsto \E_x \exp(u^\top X_T)$ is continuous on $E$.
\end{lemma}
\begin{proof}
Recall that $K$ is a transition kernel from $E$ to $F$ satisfying
$F+E\subset E$. Iterating this relation yields $nF+E\subset E$ for
all $n\in\N$. Since $\sup_{x\in E}\Re u^\top x<\infty$, it follows
that $\Re u^\top z\leq 0$ for $z\in F$. Hence $f$ given by
\[
f(x)=u^\top b(x)+\half u^\top c(x) u + \int(e^{u^\top z}-1-u^\top
z)K(x,\dd z)
\]
is well-defined and by It\^{o}'s formula
\begin{align*}
\exp(u^\top X_t)-\int_0^t \exp(u^\top X_s) f(X_s)\dd s
\end{align*}
is a local martingale. Therefore, there exists a sequence of
stopping times $T_n\uparrow\infty$ such that
\[
\E_x \exp(u^\top X_{t\wedge T_n})=\exp(u^\top x)+\E_x
\int_0^{t\wedge T_n}\exp(u^\top X_{s})f(X_{s})\dd s.
\]
Note that $|\exp(u^\top X_s)f(X_{s})|\leq C(1+\sup_{s\leq t
}|X_s|)$ for some $C>0$. In view of Lemma~\ref{lem:sup}, we can
apply the Dominated Convergence Theorem as well as Fubini's
Theorem to derive that
\begin{align}\label{eq:f}
\E_x \exp(u^\top X_{t})=\exp(u^\top x)+ \int_0^{t}\E_x(\exp(u^\top
X_{s})f(X_{s}))\dd s.
\end{align}
By the same lemma together with the Dominated Convergence Theorem
we get that $s\mapsto \E_x(\exp(u^\top X_{s})f(X_{s}))$ is
continuous, as $X_s$ is right-continuous and quasi
left-continuous. The Fundamental Theorem of Calculus yields that
$t\mapsto\E_x \exp(u^\top X_{t})$ is continuously differentiable
for all $x\in E$, which implies that $\Psi_0$ and $\psi_i$ are
continuously differentiable in $t$. Define $\psi_0$ by
\[
\psi_0(t)=\int_0^t\frac{\Psi_0'(s)}{\Psi_0(s)}\dd s,\quad t<T.
\]
Then $\psi_0$ is also continuously differentiable and
$\Psi_0(t)=\exp(\psi_0(t))$ (indeed, the quotient of the left- and
right-hand side has derivative 0 and equality holds for $t=0$,
whence it holds for all $t$). Necessarily $(\psi_0,\psi)$ has to
satisfy the generalized Riccati equations (\ref{eq:riccati}), in
view of (\ref{eq:dotpsi}).

To show the second assertion we note that by (\ref{eq:f}) and the
previous we have
\begin{align*}
\E_x(\exp(u^\top X_{t})f(X_{t}))&=\frac{\partial}{\partial t}\E_x
\exp(u^\top X_{t})\\&=(\dot{\psi_0}(t,u)+\dot{\psi}(t,u)^\top
x)\exp(\psi_0(t,u)+\psi(t,u)^\top x).
\end{align*}
So $x\mapsto \E_x(\exp(u^\top X_{t})f(X_{t}))$ is continuous for
$t<T$. By Lemma~\ref{lem:sup} and the Dominated Convergence
Theorem we see that
\[
x\mapsto \int_0^{T}\E_x(\exp(u^\top X_{s})f(X_{s}))\dd s\mbox{ is
continuous},
\]
whence $x\mapsto \E_x \exp(u^\top X_{T})$ is continuous.
\end{proof}

To extend the validity of the affine transform formula from real
to complex exponentials, we use the analyticity of the
characteristic function and the solutions to the Riccati
equations. This is demonstrated in the next lemma, which we apply
in Proposition~\ref{prop:complexexp} below to derive the desired
assertion.
\begin{lemma}\label{lem:analytic}
Consider the situation of Theorem~\ref{th:maintheorem}. For
$t\geq0$, if $U\subset S(M(t))\cap D_\C(t)$ is connected and $0\in
U$, then (\ref{eq:char}) holds for all $u\in U$.
\end{lemma}
\begin{proof}
By Proposition~\ref{prop:Expinfty} equality (\ref{eq:char}) holds
for $u\in M(t)$. The left-hand side of (\ref{eq:char}) as a
function of $u$ is analytic on $S(M(t))$ and the right-hand side
is analytic on $D_\C(t)$, see
Proposition~\ref{prop:filipmayer}~(\ref{item:fm2}) and
(\ref{item:fm3}). By assumption and the fact that $S(M(t))\cap
D_\C(t)$ is an open neighborhood of $0$ (since $M(t)=D_\R(t)$ by
Proposition~\ref{prop:Expinfty} and $D_\C(t)$ is an open
neighborhood of $0$ by
Proposition~\ref{prop:filipmayer}~(\ref{item:fm2})), there exists
an open domain $B\subset S(M(t))\cap D_\C(t)$ containing the
connected set $U\cup M(t)$ (as $M(t)$ is convex). It holds that
$M(t)$, being an open set in $\R^p$, is a set of uniqueness for
$B$, whence we can extend the equality in (\ref{eq:char}) to $u\in
B$, in particular to $u\in U$.
\end{proof}
\begin{prop}\label{prop:complexexp}
Consider the situation of Theorem~\ref{th:maintheorem} and let
$T_0>0$ be arbitrary. Then $S(D_\R(T_0))\subset D_\C(T_0)$ and the
affine transform formula (\ref{eq:char}) holds for all $u\in
S(D_\R(T_0))$, $t=T_0$.
\end{prop}
\begin{proof}
In view of Lemma~\ref{lem:analytic} it suffices to show
$S(M(T_0))\subset D_\C(T_0)$.
We argue by contradiction. Suppose there exists $u^\ast\in
S(M(T_0))\cap D_\C(T_0)^c$. We divide the proof into a couple of
steps. In the following we write $[0,u]$ for the line segment in
$\C^p$ with endpoints $0$ and $u$. For a function $f$ we write
$f([0,t])$ for the path $s\mapsto f(s)$, $s\in[0,t]$. Furthermore,
throughout we use that (\ref{eq:char}) holds for $u\in M(t)$,
which follows from Proposition~\ref{prop:Expinfty}.

\emph{Step 1.} There exists $u_0\in[0,u^\ast]$ such that
\begin{align*}
[0,u_0]&\subset S(M(t))\cap D_\C(t),\qquad\mbox{ for }t<T:=t_\infty(u_0),\\
[0,u_0)&\subset S(M(T))\cap D_\C(T).
\end{align*}
We prove this as follows. Since $S(M(T_0))$ is convex, the line
from $[0,u^\ast]$ is contained in $S(M(T_0))$. Define
\[
\lambda_0 = \inf\{\lambda\geq 0: \lambda u^\ast\not\in
D_\C(T_0)\}.
\]
Then $0<\lambda _0\leq 1$, since $D_\C(T_0)$ is an open
neighborhood of $0$, by
Proposition~\ref{prop:filipmayer}~(\ref{item:fm2}). Moreover,
$\lambda u^\ast\in D_\C(T_0)$ for $\lambda<\lambda_0$ and
$\lambda_0 u^\ast\not\in D_\C(T_0)$. Take $u_0=\lambda_0 u^\ast$,
$T=t_\infty(u_0)$. Note that $T\leq T_0$, so $D_\C(T_0)\subset
D_\C(T)$. Then by the previous we have $[0,u_0)\subset
D_\C(T_0)\subset D_\C(T)\subset D_\C(t)$, for $t<T$. Moreover,
$u_0\in D_\C(t)$ for $t<T=t_\infty(u_0)$.
This yields the assertion.

\emph{Step 2.} For all open $B\subset E^\circ$ there exists $x\in
B$ such that $\E_x\exp(u_0^\top X_{T})=0$. To see this, first note
that $\Re u_0\in M(T)\subset M(t)$ for $t\leq T$ and that
(\ref{eq:char}) holds for $u=\Re u_0$. Therefore,
\begin{align*}
\E_x\exp(\Re u_0 X_t)&=\exp(\psi_0(t,\Re u_0)+\psi(t,\Re u_0)^\top
x)\\&\rightarrow \exp(\psi_0(T,\Re u_0)+\psi(T,\Re u_0)^\top
x)=\E_x\exp(\Re u_0 X_{T})<\infty,
\end{align*}
for $t\uparrow T$, $x\in E$. By quasi-left continuity we have
$X_t\rightarrow X_{T}$, $\PP_x$-a.s. Since $|\exp(u_0^\top X_t)|$
is bounded by $\exp(\Re u_0 X_t)$ (indeed it is equal), an
extended version of the Dominated Convergence Theorem
\cite[Theorem 1.21]{Kallenberg} yields
\[
\lim_{t\uparrow T} \E_x \exp(u_0^\top X_t)=\E_x \exp(u_0^\top
X_{T}),
\]
for all $x\in E$. In particular
\[
\lim_{t\uparrow T}\exp(\psi_0(t,u_0)+\psi(t,u_0)^\top x)\mbox{
exists and is finite, for all }x\in E.
\]
Since $T=t_\infty(u_0)$, we have $\lim_{t\uparrow
T}|\psi(t,u_0)|=\infty$, by
Proposition~\ref{prop:filipmayer}~(\ref{item:fm1}). It follows
that for all open balls $B\subset E^\circ$ there exists $x\in B$
such that
\[
\E_x \exp(u_0^\top X_{T}) = 0,
\]
as otherwise $\lim_{t\uparrow T}(\psi_0(t,u_0)+\psi(t,u_0)^\top
x)$ would be finite on some ball $B$, which would give a finite
limit for $\psi(t,u_0)$, a contradiction.

\emph{Step 3.} Fix $0<\varepsilon<T$. There exists
$0<\delta<T-\varepsilon$ such that
\begin{align}\label{eq:pathinS}
\psi([0,\varepsilon+\delta],u_0)\subset S(M(T-\varepsilon)).
\end{align}
The proof is as follows. Step~1 together with
Lemma~\ref{lem:analytic} implies that (\ref{eq:char}) holds for
$u=u_0$ and $t<T$. Hence by Jensen's inequality and the Markov
property we have for $t<\varepsilon$, $x\in E$ that
\begin{align*}
\E_x \exp(\Re\psi_0(t,u_0)+\Re\psi(t,u_0) X_{T-\varepsilon}) &=
\E_x |
\exp(\psi_0(t,u_0)+\psi(t,u_0)^\top X_{T-\varepsilon})|\\
&=\E_x|\E_{X_{T-\varepsilon}} \exp(u_0^\top X_{t})|\\
&\leq \E_x \E_{X_{T-\varepsilon}}\exp(\Re u_0^\top
X_{t})\\
&= \E_x\exp(\Re u_0^\top X_{{T-\varepsilon}+t}).
\end{align*}
Since $\Re u_0\in M(T)\subset M(T-\varepsilon)$ it follows that
for $t<\varepsilon$, $x\in E$ we have
\begin{align*}
\E_x \exp(\Re\psi(t,u_0) X_{T-\varepsilon}) &\leq
\exp(-\Re\psi_0(t,u_0))\E_x\exp(\Re u_0^\top
X_{{T-\varepsilon}+t})\\&=\exp(-\Re\psi_0(t,
u_0)+\psi_0(T-\varepsilon+t,\Re
u_0)\\&\qquad\qquad\qquad\qquad\quad+\psi(T-\varepsilon+t,\Re
u_0)^\top x).
\end{align*}
Fatou's Lemma yields
\begin{align*}
\E_x \exp(\Re\psi(\varepsilon,u_0)
X_{T-\varepsilon})&\leq\liminf_{t\uparrow \varepsilon}\E_x
\exp(\Re\psi(t,u_0) X_{T-\varepsilon})\\&\leq
\exp(-\Re\psi_0(\varepsilon, u_0)+\psi_0(T,\Re u_0)+\psi(T,\Re
u_0)^\top x)\\&<\infty,
\end{align*}
for all $x\in E$. Hence $\psi([0,\varepsilon],u_0)\subset
S(M(T-\varepsilon))$. Since $S(M(T-\varepsilon))$ is open and
$t\mapsto \psi(t,u_0)$ is continuous on $[0,T)$, the result
follows.

\emph{Step 4.} It holds that $x\mapsto \E_x \exp(u_0^\top X_{T})$
is not continuous. To show this, we argue by contradiction and
assume it is continuous. Then we have $\E_x \exp(u_0^\top
X_{T})=0$ for all $x\in E$, by Step~2 and the fact that
$E=\overline{E^\circ}$. The Markov property gives
\begin{equation}\label{eq:0ismarkov}
\begin{split}
0&=\E_x \exp(u_0^\top X_{T})\\&=\E_x\E_{X_{{T}-t}}\exp(u_0^\top
X_{t})\\&=\E_x\exp(\psi_0(t,u_0)+\psi(t,u_0)^\top
X_{{T}-t}),\mbox{ for all } 0\leq t<T,x\in E,
\end{split}
\end{equation}
so $\E_x\exp(\psi(t,u_0)^\top X_{T-t})=0$ for all $0\leq t<T$,
$x\in E$. Fix $0<\varepsilon<T$ and write
$v=\psi(\varepsilon,u_0)$ and $s=T-\varepsilon$. By the semi-group
property of the flow we have $\psi(t,v)=\psi(t+\varepsilon,u_0)$
for $t<s$, whence the previous yields
\[
\E_x\exp(\psi(t,v)^\top X_{s-t})=0,\mbox{ for all $0\leq t<s$,
$x\in E$}.
\]
Let $\delta$ be as in Step~3. Then $\E_x \exp(\psi(t,v)^\top
X_{s})$ is well-defined for $t\leq \delta$, $x\in E$. Applying the
Markov property yields
\begin{equation*}
\begin{split}
\E_x\exp(\psi(t,v)^\top X_{s})&=\E_x\E_{X_t}\exp(\psi(t,v)^\top
X_{s-t}) =0,\mbox{ all $0\leq  t\leq \delta$, $x\in E$},
\end{split}
\end{equation*}
Plugging back $v=\psi(\varepsilon,u_0)$ and $s=T-\varepsilon$ and
using the semi-group property of the flow, we see that
\begin{align}\label{eq:zero}
\E_x\exp(\psi(t+\varepsilon,u_0)^\top X_{T-\varepsilon})=0,\mbox{
for all $0\leq t\leq\delta$, $x\in E$}.
\end{align}
Now fix $x\in E$. It holds that $u\mapsto \E_{x}\exp(u^\top
X_{T-\varepsilon})$ and $t\mapsto \psi(t,u_0)$ are analytic on
$S(M(T-\varepsilon))$ respectively $[0,T)$, see
Proposition~\ref{prop:filipmayer}~(\ref{item:fm2}) and
~(\ref{item:fm3}). Step~3 yields (\ref{eq:pathinS}). Therefore,
there exists an open domain $B\subset\C^p$ with
$[0,\varepsilon+\delta]\subset B$ such that $\psi(z,u_0)\in
S(M(T-\varepsilon))$ for $z\in B$. The composition of analytic
functions is analytic, whence
\[
z\mapsto \E_{x}\exp(\psi(z,u_0)^\top X_{T-\varepsilon})
\]
is analytic on $B$. Equation (\ref{eq:zero}) yields it is zero on
$[\varepsilon,\varepsilon+\delta]$, whence it is zero on the whole
of $B$, as $[\varepsilon,\varepsilon+\delta]$ is a set of
uniqueness for $B$. In particular it is zero for $z=0$, i.e.\
$\E_x\exp(u_0^\top X_{T-\varepsilon})=0$. However, by Step~1 and
Lemma~\ref{lem:analytic} we have
\[
\E_x\exp(u_0^\top
X_{T-\varepsilon})=\exp(\psi_0(T-\varepsilon,u_0)+\psi(T-\varepsilon,u_0)^\top
x )\neq 0,
\]
a contradiction.

\emph{Step 5.} It holds that $\ii\R^p\subset D_\C(T)$ and the
affine transform formula (\ref{eq:char}) holds for $u\in\ii\R^p$,
$t=T$. Indeed, if $u^\ast\in \ii\R^p$, then also $u_0\in\ii\R^p$,
as $u_0\in[0,u]$. Step~1 together with Lemma~\ref{lem:analytic}
yields (\ref{eq:char}) for $u=u_0$, $t<T$. However,
Lemma~\ref{lem:continx} then gives that
$x\mapsto\E_x(\exp(u_0^\top X_T))$ is continuous, which
contradicts Step 2. Hence $\ii\R^p\subset D_\C(T)$. By
Lemma~\ref{lem:analytic} again we get validity of (\ref{eq:char})
for $u\in\ii\R^p$, $t=T$.

\emph{Step 6.} We conclude the proof by showing that $x\mapsto
\E_x \exp(u^\top X_{T})$ is continuous for all $u\in S(M(T))$,
which contradicts Step 2. Let $x_n\rightarrow x$, some $x_n,x\in
E$. By Step 5 we have for all $u\in\ii\R^p$ that
\begin{align*}
\E_{x_n}\exp(u X_T)&=\exp(\psi_0(T,u)+\psi(T,u)^\top
x_n)\\&\rightarrow \exp(\psi_0(T,u)+\psi(T,u)^\top x)=
\E_{x}\exp(u X_T),
\end{align*}
as $n\rightarrow\infty$. Hence $\PP_{x_n}\circ
X_T^{-1}\rightarrow\PP_x\circ X_T^{-1}$ weakly. By Skorohod's
Representation Theorem \cite[Theorem~4.30]{Kallenberg} there exist
random variables $Y_n$, $Y$ defined on a common probability space
$(\Omega,\mathcal{F},P)$ such that $P\circ Y_n^{-1}=\PP_{x_n}\circ
X_T^{-1}$, $P\circ Y^{-1}=\PP_{x}\circ X_T^{-1}$ and
$Y_n\rightarrow Y$, $P$-a.s. Now let $u\in S(M(T))$ be arbitrary.
It holds that $|\exp(u^\top Y_n)|= \exp(\Re u^\top Y_n)$ and
\begin{align*}
\int \exp(\Re u^\top  Y_n)\dd P&= \exp(\psi_0(T,\Re u)+\psi(T,\Re
u)^\top {x_n})\\&\rightarrow \exp(\psi_0(T,\Re u)+\psi(T,\Re
u)^\top x)=\int \exp(\Re u^\top  Y)\dd P,
\end{align*}
for $n\rightarrow\infty$, since $\Re u\in M(T)$. An extended
version of the Dominated Convergence Theorem \cite[Theorem
1.21]{Kallenberg} yields
\begin{align*}
\E_{x_n}\exp(u^\top X_T)&= \int \exp(u^\top Y_n)\dd
P\\&\rightarrow\int \exp(u^\top  Y)\dd P=\E_{x}\exp(u^\top X_T),
\end{align*}
for $n\rightarrow\infty$, whence $x\mapsto \E_x \exp(u^\top
X_{T})$ is continuous.
\end{proof}

\section{Additional results for bounded exponentials}\label{sec:boundedexp}

In this section we relax condition (\ref{eq:exponentialmom}) of
Theorem~\ref{th:maintheorem} on the exponential moments of the
$K^i$ and consider the validity of the affine transform formula
when the left-hand side of (\ref{eq:char}) is uniformly bounded in
$t$ and $x$ (which includes the characteristic function). The
following theorem is the third main result of this paper.
\begin{theorem}\label{th:charexplo}
Suppose $E\subset\R^p$ is closed convex with non-empty interior
and let $X$ be an affine jump-diffusion on
$(D_E(0,\infty],\mathcal{F}^X,(\mathcal{F}^X_{t+}),\PP)$ with
differential characteristics $(b(X),c(X),K(X,\dd z))$ given by
(\ref{eq:bcK}). Assume (\ref{eq:z2Kfinite}) and write
 $U=\{u\in \C^p:\sup_{x\in E}\Re u^\top x<\infty\}$. Then for
all $u\in U$ there exists a $t_\infty(u)\in(0,\infty]$ and a
solution
$(\psi_0(\cdot,u),\psi(\cdot,u)):[0,t_\infty(u))\rightarrow
\C\times\C^p$ to the system of generalized Riccati equations given
by (\ref{eq:riccati}) and for all $x\in E$ it holds that
\[
\E_x\exp(u^\top X_t)=\begin{cases}\exp(\psi_0(t,u)+\psi(t,u)^\top
x),&\quad  t \in [0, t_\infty(u))\\
0,&\quad t \in [t_\infty(u),\infty)
\end{cases}
\]
\end{theorem}
\begin{proof}
For $n\in\N$ define
\[
b^n(x)=b(x)+\int z(e^{-\frac{1}{n}|z|^2}-1)K(x,\dd z),\quad
K^n(x,\dd z)=e^{-\frac{1}{n}|z|^2}K(x,\dd z),
\]
and the operator $\mathcal{A}^n:C^\infty_c(E)\rightarrow C_0(E)$
by
\begin{align*}
\mathcal{A}^nf(x)&=\nabla f(x) b^n(x)+\half\tr(\nabla^2
f(x)c(x))\\&+\int (f(x+z)-f(x)-\nabla f(x) z)K^n(x,\dd z).
\end{align*}
Then the affine martingale problem for $\mathcal{A}^n$ is
well-posed by Proposition~\ref{prop:Lismart}. Let $\Q^n_x$ be the
solution for $(\mathcal{A}^n,\delta_x)$ and write $\E^n_x$ for the
expectation with respect to $\Q^n_x$. Since $K^n$ satisfies
(\ref{eq:exponentialmom}), Theorem~\ref{th:maintheorem} yields
\[
\E^n_x\exp(u^\top X_t)=\exp(\psi^n_0(t,u)+\psi^n(t,u)^\top x),
\]
for all $u\in U$, $x\in E$, $t\geq 0$, where $(\psi_0^n,\psi^n)$
satisfies (\ref{eq:riccati}) with $b$ and $K$ replaced by $b^n$
and $K^n$. Fix $x\in E$ arbitrarily and let (\ref{eq:decompx}) be
the decomposition of $X$ under $\PP_x$. By
Proposition~\ref{prop:Lismart} it holds that
$\left.\Q^n_x\right|_{\mathcal{F}^X_{t+}}=L^n_t\cdot\left.\PP_x\right|_{\mathcal{F}^X_{t+}}$
for all $t\geq0$, where
\[
L^n = \mathcal{E}((e^{-\frac{1}{n}|z|^2}-1)\ast
(\mu^X-\nu^X))=\exp((1-\textstyle e^{-\frac{1}{n}|z|^2})\ast
\nu^X_t -\frac{1}{n}|z|^2\ast\mu^X_t).
\]
For all $u\in U$ there is a constant $C>0$ such that $|\exp(u^\top
X_t)L_t^n|\leq C L_t^n$. Since $\E_x L_t^n=1$ for all $n$ and
$\lim_{n\rightarrow\infty} L^n_t=1$, an extended version of the
Dominated Convergence Theorem \cite[Theorem 1.21]{Kallenberg}
yields
\[
\lim_{n\rightarrow\infty}\E^n_x\exp(u^\top
X_t)=\lim_{n\rightarrow\infty}\E_x\exp(u^\top
X_t)L_t^n=\E_x\exp(u^\top X_t),
\]
for all $t\geq0$, $u\in U$. Since $x\in E$ was taken arbitrarily,
this yields
\[
\E_x\exp(u^\top
X_t)=\lim_{n\rightarrow\infty}\exp(\psi^n_0(t,u)+\psi^n(t,u)^\top
x),
\]
for all $u\in U$, $x\in E$, $t\geq 0$. If $\E_x\exp(u^\top
X_t)\neq 0$ for all $u\in U$, $x\in E$, $t\geq 0$, then
$\lim_{n\rightarrow\infty}\psi^n_0(t,u)$ and
$\lim_{n\rightarrow\infty}\psi^n(t,u)$ exist and are finite for
all $t\geq0$, $u\in U$, and the result follows from
Lemma~\ref{lem:continx}.

Suppose $\E_{x_0}\exp(u^\top X_T)=0$ for some $u\in U$, $T>0$,
$x_0\in E$. We first show that then $\E_x\exp(u^\top X_T)=0$ for
all $x\in E^\circ$. If
$\limsup_{n\rightarrow\infty}|\Re\psi^n(T,u)|<\infty$, then
necessarily $\limsup_{n\rightarrow\infty}\Re\psi^n_0(T,u)=-\infty$
and the assertion follows immediately. Otherwise, there exists a
subsequence of $\psi^n$ (also denoted by $\psi^n$) and an
$i\in\{1,\ldots, p\}$ such that
\[
\lim_{n\rightarrow\infty}\Re\psi_i^n(T,u)=\pm\infty.
\]
Then if there exists $x\in E^\circ$ such that
\[
\liminf_{n\rightarrow\infty}(\Re\psi^n_0(T,u)+\Re\psi^n(T,u)^\top
x)>-\infty,
\]
then $y$ with $y_j=x_j$ for $j\neq i$ and $y_i=x_i\pm\varepsilon$
for some small $\varepsilon>0$ satisfies
\[
\liminf_{n\rightarrow\infty}(\Re\psi^n_0(T,u)+\Re\psi^n(T,u)^\top
y)=\infty.
\]
This is impossible, since $\E_y\exp(u^\top X_T)$ is finite.

Thus $\E_x\exp(u^\top X_T)=0$ for all $x\in E^\circ$. Let
$t_\infty(u)$ be given by
\[
t_\infty(u)=\inf\{t\geq0:\E_x\exp(u^\top X_t)=0\mbox{ for some
}x\in E \}.
\]
Then $t_\infty(u)>0$. Indeed, otherwise for all $t>0$ there exists
$x\in E$ and $s<t$ such that $\E_x\exp(u^\top X_s)=0$. But then
for all $t>0$ there exists $s<t$ such that $\E_x\exp(u^\top
X_s)=0$ for all $x\in E^\circ$, in view of the previous.
Right-continuity of $t\mapsto X_t$ in $0$ yields $\exp(u^\top
x)=0$ for all $x\in E^\circ$, which is absurd.

Note that $\E_x\exp(u^\top X_{t_\infty(u)})=0$ for all $x\in
E^\circ$, as $X$ is right-continuous.  For $t<t_\infty(u)$ we have
existence of finite limits for $\psi^n_0(t,u)$ and $\psi^n(t,u)$.
Lemma~\ref{lem:continx} yields (\ref{eq:char}) where
$(\psi_0,\psi)$ are solutions to the generalized Riccati equations
for $t<t_\infty(u)$. In addition it implies that $x\mapsto
\E_x\exp(u^\top X_{t_\infty(u)})$ is continuous, whence we have
$\E_x\exp(u^\top X_{t_\infty(u)})=0$ for all $x\in E$. Applying
the Markov property we see that for $t\geq t_\infty(u)$ it holds
that
\[
\E_x\exp(u^\top X_{t})=\E_x\E_{X_{t-t_\infty(u)}}\exp(u^\top
X_{t_\infty(u)})=0,
\]
which concludes the proof.
\end{proof}
Under analyticity of the Riccati functions $R_i$, we can sharpen
the assertion in Theorem~\ref{th:charexplo}.
\begin{theorem}\label{th:BsupsetU}
Consider the situation of Theorem~\ref{th:charexplo}. Assume there
exists an open domain $B\supset U$ such that
\[
\int_{\{|z|>1\}} e^{k^\top z}|K|^i(\dd z)<\infty,\mbox{ for all
}k\in  B\cap\R^p,i=0,\ldots,p.
\]
Then $t_\infty(u)=\infty$ and (\ref{eq:char}) holds for all $u\in
U$, $t\geq0$. In particular, $X$ is a regular affine process.
\end{theorem}
\begin{proof}
We argue as in Proposition~\ref{prop:complexexp}, Step~4. Let
$u_0\in U$ and suppose $T:=t_\infty(u_0)<\infty$. For $t<T$,
$u=u_0$ we have (\ref{eq:char}), which implies that
$\psi(t,u_0)\in U$, as $\E_x\exp(u_0^\top X_t)$ is bounded in $x$.
Similar as in (\ref{eq:0ismarkov}) we deduce that
$\E_x\exp(\psi(t,u_0)^\top X_{T-t})=0$ for all $0\leq t<T$, $x\in
E$. Fix $0<\varepsilon<T$ and write $v=\psi(\varepsilon,u_0)$ and
$s=T-\varepsilon$. We have $\psi(t,v)\in U$, so
$\E_x\exp(\psi(t,v)X_s)$ is well-defined for $t<s$, $x\in E$. By
the same argument as in Step~4 of
Proposition~\ref{prop:complexexp}, we get (\ref{eq:zero}), with
$\delta<T-\varepsilon$. Since $\psi(t,u_0)\in U\subset B$ for all
$t<T$ and $R_i$ given by (\ref{eq:Ri}) is analytic on $B$, it
follows by standard ODE results (e.g.\ \cite[Theorem
10.4.5]{dieu69}) that $t\mapsto\psi(t,u_0)$ is analytic on
$[0,T)$. Moreover, for all $u\in B$ there exists a solution
$(\psi_0,\psi)$ to (\ref{eq:riccati}) on a non-empty interval
$[0,t_\infty(u))$ with
\[
t_\infty(u)=\lim_{n\rightarrow\infty}\inf\{t\geq 0:\psi(t,u)\in
\partial B\mbox{ or }|\psi(t,u)|\geq n\},
\]
and
\[
D(t)=\{u\in B:t<t_\infty(u)\}
\]
is an open set containing $U$, for all $t\geq0$, see
\cite[Theorems 7.6 and 8.3]{amann90}. Theorem~\ref{th:ricexists}
implies that (\ref{eq:char}) holds for $u\in D(t)\cap\R^p$ for all
$t\geq0$. By Proposition~\ref{prop:filipmayer}~(\ref{item:fm3}) we
obtain that $u\mapsto \E_x\exp(u^\top X_t)$ is analytic on $U$ for
$x\in E$, for all $t\geq0$. It follows that
$\E_x\exp(\psi(t,u_0)^\top X_{T-\varepsilon})$ is analytic in $t$.
Since it is zero on $[\varepsilon,\varepsilon+\delta]$, it is zero
everywhere, in particular it is zero at $t=0$. This contradicts
the fact that $T-\varepsilon<t_\infty(u_0)$.
\end{proof}

\subsection{Infinite divisibility}\label{subsec:infdiv}


As a corollary of Theorem~\ref{th:charexplo} we obtain a
sufficient criterium for infinite divisibility of an affine
jump-diffusion with a general closed convex state space.
\begin{theorem}
Consider the situation of Theorem~\ref{th:charexplo}. Suppose for
all $n\in\N$ it holds that
\begin{align}\label{eq:admis}
(\textstyle a^i,nA^i,\frac{1}{n}K^i(\frac{1}{n}\dd z))_{0\leq
i\leq p}
\end{align}
is an admissible parameter set. Then $\PP_x\circ X_t^{-1}$ is
infinitely divisible for all $t\geq 0$, $x\in E$. Consequently,
$t_\infty(u)=\infty$ and (\ref{eq:char}) holds for all $u\in U$,
$t\geq0$.
\end{theorem}
\begin{proof}
Let $(\psi_0,\psi)$ be the solution to the Riccati equations as
given in Theorem~\ref{th:charexplo}. Define
$\psi^n_i=\frac{1}{n}\psi_i$, for $i=0,\ldots,p$. Then
$(\psi^n_0,\psi^n)$ solve the system of Riccati equations
corresponding to an affine jump-diffusion with parameter set
(\ref{eq:admis}). Let $\PP^n_x$ be the solution of the associated
affine martingale problem with initial condition $\delta_x$ and
write $\E^n_x$ for the expectation with respect to this
probability measure. From Theorem~\ref{th:charexplo} it follows
that
\[
(\E_x \exp(u X_t))^{1/n}=\E_x^n \exp(u X_t),
\]
for all $x\in E$, $u\in U$. In particular it holds for
$u\in\ii\R^p$, which yields the result.
\end{proof}

\subsection{Self-dual cone}\label{subsec:selfdual}


We can strengthen the conditions of Theorem~\ref{th:BsupsetU} in
case $E$ is a \emph{self-dual cone}. Recall that $E$ is a
self-dual cone with respect to an inner product
$\langle\cdot,\cdot\rangle$ if
\[
E=\{x\in\R^p:  \langle x, y\rangle \geq 0 \mbox{ for all }y\in
E\}.
\]
In that case we also have
\[
E^\circ=\{x\in\R^p: \langle x, y\rangle  > 0 \mbox{ for all }y\in
E\backslash\{0\}\}.
\]
For $x,y\in \R^p$ we write $x\preceq y$ if $y-x\in E$ and $x\prec
y$ if $y-x\in E^\circ$. An inner product on $\R^p$ can always be
written as $\langle x,y\rangle = x^\top M y$ for some positive
definite matrix $M$. By applying the linear transformation
$x\mapsto M^{1/2} x$ on the state space $E$, we may assume without
loss of generality that the underlying inner product is the usual
Euclidean inner product and we write $x^\top y$ instead of
$\langle x,y\rangle$.

Part of the following proposition extends
\cite[Proposition~3.4]{kell09} and \cite[Lemma~3.3]{cfmt09} from
the state spaces $\R^p_+$ and $S^p_+$ to general self-dual cones.
We adapt their proofs slightly by using the analyticity of
$t\mapsto \psi_i(t,u)$ in a neighborhood of $0$ for $u\in
-E^\circ$, which is a consequence of Theorem~\ref{th:charexplo}.
\begin{prop}\label{prop:selfdualcone}
Consider the situation of Theorem~\ref{th:charexplo}. Assume the
state space $E$ is a self-dual cone and in addition assume
$E^\circ\subset\{x\in\R^p:\Phi(x)>0\}$ and $\partial
E\subset\{x\in\R^p:\Phi(x)=0\}$, for some analytic function
$\Phi:\R^p\rightarrow\R$. Then
for $U=-E+\ii\R^p$ it holds that
\begin{enumerate}
\item $\psi_0(t,u)\leq\psi_0(t,v)$ and $\psi(t,u)\preceq \psi(t,v)$ for $u\preceq v$ with $u,v\in \Re U$, $t\geq 0$;
\item $\psi(t,u)\in \Re U^\circ$ for
all $u\in \Re U^\circ$, $t\geq 0$;
\item $t_\infty(u)=\infty$ for $u\in U^\circ$ and $\psi(t,u)\in U^\circ$ for
all $u\in U^\circ$, $t\geq 0$;
\end{enumerate}
\end{prop}
\begin{remark}
Examples of such state spaces are $\R^p_+$, $\vech(S^p_+)$ (with
inner product $\langle x,y\rangle=\tr(
\vech^{-1}(x)\vech^{-1}(y))$) and the Lorentz cone $\{x\in\R^p:
x_1\geq (\sum_{i=2}^p x_i^2)^{1/2}\}$. The analytic function
$\Phi(x)$ can be chosen to be respectively $\Phi(x)=\prod_{i=1}^p
x_i$, $\Phi(x)=\det(\vech^{-1}(x))$ and
$\Phi(x)=x_1^2-\sum_{i=1}^p x_i^2$.
\end{remark}
\begin{proof}
If $u\preceq v$ with $u,v\in \Re U$, then $u^\top x \leq v^\top x$
for all $x\in E$. Hence
\[
\E_x \exp(u^\top X_t)\leq \E_x \exp(v^\top X_t),\mbox{ for all
$x\in E$, $t\geq0$}.
\]
Since the affine transform formula is valid for $u\in\Re U$ by
Theorem~\ref{th:charexplo}, it follows that
\[
\psi_0(t,u)+\psi(t,u)^\top x \leq \psi_0(t,v)+\psi(t,v)^\top
x,\mbox{ for all $x\in E$, $t\geq0$}.
\]
Taking $x=nx_0$ with $x_0\in E\backslash\{0\}$, $n\in\N$ and
letting $n$ tend to infinity, we obtain the first assertion.

We prove the second assertion from an argument by contradiction.
Suppose $u_0\in \Re U^\circ$ and suppose
\[
t:=\inf\{s>0:\psi(s,u_0)\not\in \Re U^\circ\}<\infty.
\]
Then $\psi(t,u_0)\in \partial E$, so $\psi(t,u_0)^\top x_0=0$ for
some $x_0\in E$. If $u_0\preceq v$, then $\psi(t,u_0)\preceq
\psi(t,v)$ by the first assertion and since $\psi(t,v)\in -E$ we
have $\psi(t,u_0)^\top x\leq\psi(t,v)^\top x\leq 0$ for all $x\in
E$. Hence $\psi(t,v)^\top x_0=0$ and $\psi(t,v)\in \partial E$.
Thus we have
\[
\Phi(\psi(t,v))=0,\mbox{ for all }v\succeq u_0.
\]
It holds that $\{v\in \Re U^\circ:v\succeq u_0\}$ is a set of
uniqueness. Moreover, $u\mapsto \E_x \exp(u^\top X_t)$ is analytic
on $U^\circ$ for all $x\in E$, by
Proposition~\ref{prop:filipmayer}~(\ref{item:fm3}). This implies
that $u\mapsto\psi(t,u)$ is analytic on $\Re U^\circ$. It follows
that
\[
\Phi(\psi(t,u))=0,\mbox{ for all }u\in \Re U.
\]
In particular (take $u=\psi(s,u_0)$) we have
\[
\Phi(\psi(t+s,u_0)=\Phi(\psi(t,\psi(s,u_0)))=0,\mbox{ for all
$s>0$}.
\]
Let $\varepsilon>0$ be such that $\psi(s,u)\in \Re U^\circ$ for
$-\varepsilon<s<\varepsilon$. Then $s\mapsto \psi(s,u)$ is
analytic on $(-\varepsilon,\varepsilon)$ in view of
(\ref{eq:riccati}) and the analyticity of (\ref{eq:Ri}). Hence
$s\mapsto \Phi(\psi(t+s,u))$ is analytic on
$(-\varepsilon,\varepsilon)$ and it follows that it is zero on
this interval, as it is zero on $[0,\varepsilon)$. This
contradicts $\psi(s,u_0)\in\Re U^\circ$ for $s<t$.

For the third assertion, let $u\in U^\circ$. Then
\begin{align*}
\exp(\Re\psi_0(t,u)+\Re\psi(t,u)^\top  x)&=|\E_x(\exp(u^\top
X_t))|
\\&\leq \E_x(\exp(\Re u^\top X_t)\\&=\exp(\psi_0(t,\Re u)+\psi(t,\Re
u)^\top x),
\end{align*}
for all $x\in E$, $t<t_\infty(u)$. Take $x_0\in E\backslash\{0\}$
and $x= n x_0$ for $n\in \N$ and let $n$ tend to infinity. Then
the right-hand side of the above display tends to zero, which
implies $\Re\psi(t,u)^\top  x<0$ for all $x\in E\backslash\{0\}$,
i.e.\ $\psi(t,u)\in U^\circ$. The proof of $t_\infty(u)=\infty$
goes along the same lines as the proof of
Theorem~\ref{th:BsupsetU}.
\end{proof}

\begin{cor}
Consider the situation of Proposition~\ref{prop:selfdualcone}.
Write $K$ for the vector of signed measures $K^i$, $i=1,\ldots,p$,
let $L_u=\{z\in E:  u^\top z\neq 2k\pi,\mbox{ for all }k\in\Z\}$
and assume $K(L_u)\succ 0$ for all $u\in\R^p$. Then $t_\infty(\ii
u)=\infty$ for all $u\in\R^p$, whence $X$ is a regular affine
process in the sense of Definition~\ref{def:affine}.
\end{cor}
\begin{proof}
For $u=0$ there is nothing to prove. Let $u\in\R^p\backslash\{0\}$
be arbitrary. It suffices to prove $\dot{\psi}(0,\ii u)\in
U^\circ$. Indeed, by continuity we then have $\dot{\psi}(t,\ii
u)\in U^\circ$ for $t>0$ small enough. Hence $\psi(t,\ii u)=\ii
u+\int_0^t \dot{\psi}(s,\ii u)\dd s \in U^\circ$ for $t>0$ small
enough. The result then follows from
Proposition~\ref{prop:selfdualcone}.

We first show that $c(x)-A^0$ is positive semi-definite and $K(\dd
z) ^\top x$ is a positive measure, for all $x\in E$. Since $E$ is
a cone we have $nx\in E$, for all $n\in \N$, $x\in E$. We can
write
\[
c(nx)=A^0+n(c(x)-A^0),\quad K(nx,\dd z)=K^0(\dd z)+n K(\dd z)^\top
x.
\]
Since $c(x)$ is positive semi-definite and $K(x,\dd z)$ is a
positive measure for all $x\in E$, we have the same properties for
$c(x)-A^0$ respectively $K(\dd z)^\top x$, in view of the above
display.

Next we note that $\int (\cos(u^\top z)-1)K(\dd z)\prec 0$.
Indeed, by the assumption $K(L_u)\succ 0$ and the fact that
$f(z):=\cos( u^\top z)-1<0$ for $z\in L_u$ we have
\[
\int f(z) K(\dd z)^\top x = \int_{E\backslash L_u} f(z) K(\dd
z)^\top x + \int_{L_u} f(z) K(\dd z)^\top x <0,
\]
for $x\in E\backslash\{0\}$.

Now let $x\in E\backslash\{0\}$ be arbitrary. Then the previous
together with (\ref{eq:Ri}) yields
\[
\Re\dot{\psi}(0,\ii u)^\top x= -\half u^\top (c(x)-A^0) u + \int
(\cos(u^\top z)-1)K(\dd z)^\top x<0,
\]
whence $\Re\dot{\psi}(0,\ii u)^\top\prec 0$, as we needed to show.
\end{proof}

\appendix

\section{}

\begin{proof}[Proof of Remark~\ref{remark} part 2]
Let $f_k$ be a sequence in $C^\infty_c$ with $0\leq f_k\leq 1$ and
$f_k=1$ on the ball with center 0 and radius $k$. We define
\begin{align*}
\mathcal{C}=\{f\in C_c^\infty(E):f(x)&=\cos(u^\top x)f_k(x)\mbox{
or }\\f(x)&=\sin(u^\top x)f_k(x),\mbox{ for some $u\in\Q$,
$k\in\N$}\}.
\end{align*}
Then $P$ is a solution of the martingale problem for $\mathcal{A}$
on $\Omega$ if and only if
\[
f(X_t)-f(X_0)-\int_0^t \mathcal{A}f(X_s)\dd s
\]
is an $((\mathcal{F}^X_t),\PP)$-martingale for all
$f\in\mathcal{C}$. Indeed, suppose the latter holds, then
following the proof of \cite[Proposition~3.2]{cherfilyor} we
deduce that
\begin{equation}\label{eq:ito}
\begin{split}
f(X_t)-f(X_0)-&\int_0^t\nabla f(X_s) b(X_s)+\half\tr(\nabla^2
f(X_s)c(X_s))\\&+\int (f(X_s+z)-f(X_s)-\nabla f(X_s) z)K(X_s,\dd
z)\dd s
\end{split}
\end{equation}
is an $((\mathcal{F}^X_{t+}),\PP)$-local martingale for
$f(x)=e^{\ii u^\top x}$, for all $u\in\Q^p$, whence for all
$u\in\R^p$ by dominated convergence.
\cite[Theorem~II.2.42]{JacShir} yields that
$\PP$ is a solution of the martingale problem for $\mathcal{A}$ on
$\Omega$.

Applying \cite[Theorem~4.4.6]{ethier86} to the operator
$\left.\mathcal{A}\right|_\mathcal{C}$ gives that $x\mapsto
\PP_x(B)$ is measurable for all Borel sets $B$, i.e.\
$(\PP_x)_{x\in E}$ is a transition kernel. We note that although
we don't have well-posedness for all initial values in the sense
of \cite[Theorem~4.4.6]{ethier86}, the assertion in that theorem
still holds under the weaker assumption of well-posedness for
degenerate initial distributions. This is a consequence of the
fact that the set $\{P\in\mathcal{P}(E):P\mbox{ is degenerate}\}$
is measurable with respect to the Borel $\sigma$-algebra induced
by the Prohorov metric (in fact, it is even a closed set).

Following the last part of the proof of \cite[Theorem~21.10]
{Kallenberg} we see that $\PP_\lambda:=\int\PP_x\lambda(\dd x)$ is
the unique solution for $(\mathcal{A},\lambda)$. The strong Markov
property is a consequence of \cite[Theorem~4.4.2(c)]{ethier86}.
\end{proof}

\begin{lemma}\label{lem:appendix}
Let $\mathcal{A}$ and $\widetilde{\mathcal{A}}$ be given by
(\ref{eq:opA}) and (\ref{eq:opwA}) and assume (\ref{eq:condbcK}),
(\ref{eq:condgrowK}), (\ref{eq:condcontinu}) and
(\ref{eq:condmoregrowK}). Then for all $f\in C_c^\infty(E)$ it
holds that $\mathcal{A}f\in B(E)$ and $\widetilde{\mathcal{A}}f\in
C_0(E)$.
\end{lemma}
\begin{proof}
Take $f\in C^\infty_c(E)$ with $f(x)=0$ for $|x|> M$, some $M>0$.
Then for $|x|>M+1$ it holds that
\begin{align*}
|\mathcal{A}f(x)|&=|\int  f(x+z)K(x,\dd z)|\\&\leq
\|f\|_\infty\int_{\{|z|\geq|x|-M\}}|z|^q/(|x|-M)^q K(x,\dd z)\\
&\leq \|f\|_\infty C(1+|x|)^q/(|x|-M)^q,
\end{align*}
which is bounded for $x\geq M+1$. Hence $\mathcal{A}f\in B(E)$ and
likewise one can show that $\widetilde{\mathcal{A}}f(x)\rightarrow
0$ if $|x|\rightarrow\infty$. It remains to show that
$\widetilde{\mathcal{A}}f$ is continuous.

Write $g(x)=f(x+z)-f(x)-\nabla f(x) z$ and
\begin{align*}
\int g(x)\wK(x,\dd z)-\int g(y)\wK(y,\dd z)&= \int
(g(x)-g(y))\wK(x,\dd z)\\&+\int g(y)(\wK(x,\dd z)-\wK(y,\dd z)).
\end{align*}
The integrand of the first term on the right-hand side equals
(where $f_{ijk}$ is short-hand notation for
$\partial_i\partial_j\partial_k f$)
\begin{align*}
&\sum_{i,j,k}\int_0^1\!\!\int_0^1\!\!\int_0^1
 f_{ijk}((1-u)y+u x + st z)(x_i-y_i)s t z_j z_k   \,\dd u \,\dd s \,\dd t
\,1_{\{|z|\leq 1\}}\\
&+\sum_{i,j}\int_0^1\!\!\int_0^1 ( f_{ij}((1-t)y+t x + s
z)\\&\qquad\qquad\qquad- f_{ij}((1-t)y+t x ))(x_i-y_i)s t z_j
\,\dd s \,\dd t \,1_{\{|z|> 1\}},
\end{align*}
whence its integral tends to zero for $x\rightarrow y$ since $\int
(|z|^2\wedge|z|)\wK(\cdot,\dd z)$ is bounded on compacta. The
integrand in the second term on the right-hand side can be bounded
by a constant times $|z|^2\wedge|z|$, whence the integral tends to
zero by weak continuity of $x\mapsto(|z|^2\wedge|z|)\wK(x,\dd z)$.
It now easily follows that $\widetilde{\mathcal{A}}f$ is
continuous.
\end{proof}

\begin{lemma}\label{lem:sup}
Let $\Omega=D_E[0,\infty)$ and suppose $X$ is a special
jump-diffusion on
$(\Omega,\mathcal{F}^X,(\mathcal{F}^X_{t+}),\PP)$ with
decomposition (\ref{eq:decompx}) and differential characteristics
$(b(X) 1_{[0,\tau]},c(X)1_{[0,\tau]},K(X,\dd z)1_{[0,\tau]})$ for
some $(\mathcal{F}^X_{t+})$-stopping time $\tau$. Assume $\E
|X_0|^2<\infty$ and
\begin{align}\label{eq:lingrow}
|b(x)|^2+|c(x)|+ \int  |z|^2 K(x,\dd z)\leq C(1+|x|^2),\mbox{ for
some $C>0$, all $x\in E$}.
\end{align}
Then for all $T\geq 0$ it holds that
\[
\E \sup_{t\leq T} |X_t|^2\leq (4\E |X_0|^2+C(T)) e^{C(T)T},
\]
with $C(T)$ a constant depending on $C$ and $T$. In addition,
$X^c$ and $z\ast(\mu^X-\nu^X)$ are proper martingales.
\end{lemma}
\begin{proof}
Define stopping times $T_n=\inf\{t\geq0:|X_t|\geq n\mbox{ or
}|X_{t-}|\geq n\}$. It holds that
\[
\langle X^c\rangle_{t}^{T_n}=\int_0^{t\wedge T_n}
c(X_s)1_{[0,\tau]}(s)\dd s
\]
and
\[
\langle z\ast(\mu^X-\nu^X )\rangle^{T_n}_t=\int_0^{t\wedge
T_n}\int |z|^2 K(X_s,\dd z)1_{[0,\tau]}(s)\dd s
\]
have finite expectation, as they are bounded. This yields that
both $(X^c)^{T_n}$ and $z \ast(\mu^X-\nu^X)^{T_n}$ are
martingales, by \cite[Proposition~I.4.50]{JacShir}. For $t\geq 0$
write $\|X\|_t=\sup_{s\leq t}|X_s|$ and let $T>0$ be fixed. Then
it holds that
\begin{align*}
\quart\|X^{\tau_n}\|_T^2&\leq |X_0|^2+\sup_{t\leq
T}\left|\int_0^{t} b(X_s)1_{[0,\tau\wedge T_n]}(s)\dd s\right|^2
+\sup_{t\leq T}|X^c_{t\wedge T_n}|^2\\&+\sup_{t\leq T}\left|z\,
1_{[0,T_n]}\ast(\mu^X-\nu^X )_t\right|^2.
\end{align*}
Cauchy-Schwarz gives
\begin{align*}
\sup_{t\leq T}\left|\int_0^{t} b(X_s)1_{[0,\tau\wedge T_n]}(s)\dd
s\right|^2&\leq T\int_0^{T} |b(X_s)|^21_{[0,\tau\wedge T_n]}(s)\dd
s\\&\leq CT\int_0^{T} (1+\|X^{T_n}\|_s^2)\dd s.
\end{align*}
Doob's inequality gives
\begin{align*}
\E \sup_{t\leq T}|X^c_{t\wedge T_n}|^2&\leq 4 \E (X^c_{T\wedge
T_n})^2\leq 4\E \int_0^{T} c(X_{s})1_{[0,\tau\wedge T_n]}\dd s
\\&\leq 4C\int_0^T(1+\E\|X^{T_n}\|_s^2 )\dd s
\end{align*}
and
\begin{align*}
\E \sup_{t\leq T}\left|z\, 1_{[0,T_n]}\ast(\mu^X-\nu^X
)_t\right|^2&\leq
4\E\int_0^{T\wedge\tau\wedge T_n}\int  |z|^2\nu^X (\dd s,\dd z)\\
&\leq 4C\int_0^T(1+\E\|X^{T_n}\|_s^2 )\dd s,
\end{align*}
It follows that for $t\leq T$ we have
\[
\E\|X^{T_n}\|_t^2\leq 4|X_0|^2+C'(T)(1+\int_0^t
\E\|X^{T_n}\|_s^2\dd s),
\]
with $C'(T)$ a constant depending on $C$ and $T$. Since
\begin{align*}
\E\|X^{T_n}\|_T^2 &\leq \E|X_0|^2 + n^2+\E|\Delta X_{T\wedge T_n}|^2\\
&\leq \E|X_0|^2 +n^2+\E\int_0^{T\wedge{T_n}}\int |z|^2\mu^{X}(\dd
t,\dd z)\\
&=\E|X_0|^2 +n^2+\E\int_0^{T\wedge{T_n}}\int |z|^2\nu^X(\dd
t,\dd z)\\
&\leq \E|X_0|^2 +n^2+C\E\int_0^{T\wedge T_n}(1+|X_{s}|^2)\dd
s<\infty,
\end{align*}
Grownwall's lemma yields
\[
\E\|X^{T_n}\|_T^2\leq (4\E|X_0|^2+C(T)) e^{C(T)T}.
\]
for some constant $C(T)$ depending on $C$ and $T$. Let
$n\rightarrow\infty$, then the left-hand side converges by the
Monotone Convergence Theorem to $\E\|X\|_T^2$, which is bounded by
the right-hand side. This yields the first assertion of the lemma.
The second assertion is an immediate consequence in view of
\cite[Proposition~I.4.50]{JacShir}, since
\[
\langle X^c\rangle_{t}=\int_0^{t\wedge\tau} c(X_s)\dd s\,\mbox{
and }\, \langle z\ast(\mu^X-\nu^X
)\rangle_t=\int_0^{t\wedge\tau}\int |z|^2 K(X_s,\dd z)\dd s
\]
have finite expectation due to the growth-condition
(\ref{eq:lingrow}) and the derived moment inequality for
$|X_t|^2$.
\end{proof}

\bibliographystyle{imsart-nameyear}
\bibliography{refs}

\end{document}